\documentclass[a4paper,10pt]{article}
\usepackage{amsmath,amssymb,amsthm,mathrsfs,dsfont,mathtools,euscript,authblk,ushort,scalerel,relsize,url,nicefrac} 
\usepackage{hyperref} 
\usepackage{enumerate}
\usepackage[shortlabels]{enumitem} 

\usepackage{xcolor}
\definecolor{unbleu}{rgb}{0.03, 0.15, 0.4}
\hypersetup{
pdfborder = {0 0 0},
colorlinks,
linkcolor=unbleu,
citecolor=unbleu,
urlcolor=unbleu
}

\numberwithin{equation}{section}  
\newtheorem{theorem}{Theorem}[section]

\newtheorem{lemma}[theorem]{Lemma}
\newtheorem{proposition}[theorem]{Proposition}
\newtheorem{corollary}[theorem]{Corollary}
\newtheorem{remark}{Remark}[section]
\newtheorem*{theorem*}{Theorem}

\newtheorem*{notations}{Notations}

\newcommand{\cte}[1]{C_{\mathrm{(\ref{#1})}}}
\newcommand{\pcte}[1]{c_{\mathrm{(\ref{#1})}}}

\newcommand{\proba}{{\mathds{P}}}
\newcommand{\E}{{\mathds{E}}}
\newcommand{\real}{\mathds{R}}

\newcommand{\zentiers}{{\mathds{Z}}}
\newcommand{\integers}{\mathds{N}}
\newcommand{\domaine}{{\zentiers}_{\scaleto{+}{5pt}}^{\hspace{-.3pt}\scaleto{d}{5pt}}}
\newcommand{\sdomaine}{{\zentiers}_{\scaleto{+}{3.5pt}}^{\hspace{-.2pt}\scaleto{d}{3.5pt}}}
\newcommand{\rpd}{\mathds{R}_{\scaleto{+}{5pt}}^{\hspace{-.3pt}\scaleto{d}{5pt}}}
\newcommand{\srpd}{\mathds{R}_{\scaleto{+}{3.5pt}}^{\hspace{-.2pt}\scaleto{d}{3.5pt}}}
\newcommand{\Oun}{\mathcal{O}(1)} 
\newcommand{\Thetaun}{\Theta(1)} 
\newcommand{\sK}{{\scriptscriptstyle K}}
\newcommand{\gen}{\mathscr{L}_{\sK}}
\newcommand{\un}{\mathds{1}}
\newcommand{\Sp}{\mathrm{Sp}}
\newcommand{\vecn}{\ushort{n}}
\newcommand{\vecnf}{\ushort{n}^{\scaleto{*}{3pt}}}
\newcommand{\svecnf}{\ushort{n}^{\hspace{-1pt}\scaleto{*}{2.2pt}}}
\newcommand{\vecxf}{\ushort{x}^{\scaleto{*}{3pt}}}
\newcommand{\svecxf}{\ushort{x}^{\scaleto{*}{2pt}}}
\newcommand{\vecx}{\ushort{x}}
\newcommand{\vecy}{\ushort{y}}
\newcommand{\vecz}{\ushort{z}}
\newcommand{\vecv}{\ushort{v}}
\newcommand{\vecm}{\ushort{m}}
\newcommand{\vecX}{\ushort{X}}
\newcommand{\vecB}{\ushort{B}}
\newcommand{\vecD}{\ushort{D}}
\newcommand{\vp}{\ushort{p}}

\newcommand{\vmu}{\ushort{\mu}}
\newcommand{\veczer}{\ushort{0}}
\newcommand{\vece}[1]{\scaleto{\ushort{e}^{(#1)}}{11.5pt}}
\newcommand{\res}{\rho^{\scaleto{*}{2.8pt}}}
\newcommand{\sDelta}{\scaleto{\Delta}{5pt}}
\newcommand{\ssDelta}{\scaleto{\Delta}{3.5pt}}
\newcommand{\leproc}{\ushort{N}^{\hspace{-1pt}{\scriptscriptstyle K}}}
\newcommand{\sleproc}{\ushort{\scaleto{N}{5pt}}^{\hspace{-0.5pt}{\scaleto{K}{3pt}}}}
\newcommand{\nuk}{\nu\!_{\scaleto{K}{3.5pt}}}
\newcommand{\snuk}{\nu\!_{\scaleto{K}{2.8pt}}}
\newcommand{\e}{\operatorname{e}} 
\newcommand{\ii}{\operatorname{i}} 
\newcommand{\dd}{\hspace{1pt}\mathrm{d}}
\newcommand{\norm}[1]{\ensuremath{\scaleto{\|}{9pt}\hspace{.5pt} #1\scaleto{\|}{9pt}_{\scaleto{1}{4.pt}}}}
\newcommand{\snorm}[1]{\ensuremath{\scaleto{\|}{8pt}\hspace{.5pt} #1\scaleto{\|}{8pt}_{\scaleto{1}{3pt}}}} 
\newcommand{\normdeux}[1]{\ensuremath{\scaleto{\|}{9pt}\hspace{.5pt} #1\scaleto{\|}{9pt}_{\scaleto{2}{3.5pt}}}}
\newcommand{\snormdeux}[1]{\ensuremath{\scaleto{\|}{8pt}\hspace{.5pt} #1\scaleto{\|}{8pt}_{\scaleto{2}{3pt}}}} 

\title{Quasi-Stationary Distributions and Resilience: What to get from a sample?
}

\author[1]{J.-R. Chazottes
\thanks{Email: \texttt{chazottes@cpht.polytechnique.fr}}}
\author[1]{P. Collet
\thanks{Email: \texttt{collet@cpht.polytechnique.fr}}}
\author[2]{S. M\'el\'eard
\thanks{Email: \texttt{sylvie.meleard@polytechnique.edu}}}

\author[3]{S. Mart\'{\i}nez
\thanks{Email: \texttt{smartine@dim.uchile.cl}}}

\affil[1]{CPHT, CNRS, Ecole Polytechnique, Institut Polytechnique de Paris, F-91128 Palaiseau, France}
\affil[2]{CMAP, CNRS, Ecole Polytechnique, Institut Polytechnique de Paris, F-91128 Palaiseau, France}
\affil[3]{DIM-CMM, Universidad de Chile, UMI 2807 UChile-CNRS, Beauchef 851, Santiago, Chile}

\begin{document}

\maketitle

\begin{abstract}
We study a class of multi-species birth-and-death processes going almost surely to extinction and 
admitting a unique quasi-stationary distribution (qsd for short). When rescaled by $K$ and in the limit $K\to+\infty$,
the realizations of such processes get close, in any fixed finite-time window, to the trajectories of a dynamical system whose 
vector field is defined by the birth and death rates.
Assuming  this dynamical system has a unique attracting fixed point, we analyzed the behavior of these processes
for finite $K$ and finite times, ``interpolating" between the two limiting regimes just mentioned. 
In the present work, we are mainly interested in the following question:
Observing a realization of the process, can we determine the so-called engineering resilience? 
To answer this question, we establish two relations which intermingle the resilience, which is a 
macroscopic quantity defined for the dynamical system, and the fluctuations of the process, which are 
microscopic quantities. Analogous relations are well known in nonequilibrium statistical mechanics.
To exploit these relations, we need to introduce several estimators which we control for times between $\log K$ 
(time scale to converge to the qsd) and $\exp(K)$ (time scale of mean time to extinction).
\end{abstract}


\tableofcontents

\newpage


\textbf{Titre en fran\c{c}ais :}

Distributions quasi-stationnaires et r\'esilience : que peut-on obtenir des donn\'ees ?

\bigskip

\textbf{R\'esum\'e en fran\c{c}ais :}

Nous \'etudions une classe de processus de naissance et mort avec plusieurs esp\`eces dans la situation o\`u
l'extinction est certaine et la distribution quasi-stationnaire est unique. Si on fixe un intervalle de temps fini et qu'on normalise 
les r\'ealisations d'un tel processus par un param\`etre d'\'echelle $K$, elles deviennent arbitrairement proches, dans la limite
$K\to+\infty$, des trajectoires d'un certain syst\`eme dynamique dont le champ de vecteurs est d\'efini \`a partir des taux de 
naissance et mort.
Quand le syst\`eme dynamique admet un seul point fixe attractif, nous avons pr\'ec\'edemment analys\'e le comportement
du processus pour des valeurs de $K$ finies et pour des temps finis, c'est-\`a-dire, le comportement interm\'ediaire entre les 
deux comportements limites \'evoqu\'es ci-dessus.
La question principale qui nous int\'eresse dans le pr\'esent article est la suivante : si on observe une r\'ealisation du
processus, pouvons-nous estimer la r\'esilience au sens de l'ing\'enieur (\emph{engineering resilience}) ?
Pour r\'epondre \`a cette question, nous d\'emontrons deux relations entrem\^elant la r\'esilience, qui est une quantit\'e 
macroscopique d\'efinie pour le syst\`eme dynamique sous-jacent, et les fluctuations du processus, qui sont elles des quantit\'es
microscopiques. De tels genres de relations sont bien connues en m\'ecanique statistique hors d'\'equilibre.
Afin d'exploiter ces relations nous introduisons plusieurs estimateurs empiriques que nous parvenons \`a contr\^oler
pour des temps entre $\log K$, qui est l'\'echelle de temps pour observer la convergence vers la distribution quasi-stationnaire,
et $\exp(K)$, qui est l'\'echelle du temps moyen d'extinction.

\bigskip

\textbf{Mots cl\'es :}

processus de naissance et mort, syst\`emes dynamiques, r\'esilience, distribution quasi-stationnaire, relation de fluctuation-dissipation, estimateurs empiriques.

\newpage


\section{Introduction and main results}\label{sec:intro}

\subsection{Context and setting}

The ability of an ecosystem to return to its reference state after a perturbation stress is given by its resilience, a concept 
pioneered by Holling. Resilience has several faces and multiple definitions \cite{resilience-book}. In the traditional theoretical 
setting of dynamical systems, that is, differential equations, one of them is the so-called {\em engineering resilience}. It 
is concerned with what happens in the vicinity of a fixed point (equilibrium state) of the system, and is given by minus the real 
part of the dominant eigenvalue of the Jacobian matrix evaluated at the fixed point. It can also be defined as the reciprocal of 
the characteristic return time to the fixed point after a (small) perturbation. 
In this paper, we are interested in how to determine the engineering resilience from the data. But which data?
The drawback of the notion of engineering resilience is that we do not observe population densities governed by differential equations. 
Instead, we count individuals which are subject to stochastic fluctuations. Can we nevertheless infer the resilience? The subject of this 
paper is to show that this is possible in the framework of birth-and-death processes which are, in a sense made precise below, 
close to the solutions of a corresponding differential equation, at certain time and population size scales.

Let us now describe our framework. We consider a population made of $d$ species interacting with one another. Suppose that the state of the 
process at time $t$, which we denote by $\leproc(t)=(N_{\scaleto{1}{4.5pt}}^{\sK}(t), \ldots,N_{\scaleto{d}{4.5pt}}^{\sK}(t))$, is $\vecn=(n_{\scaleto{1}{4pt}},
\ldots,n_i,\ldots,n_{\scaleto{d}{4pt}}) \in\domaine$, where $n_i$ is the number of individuals of the $i$th species. Then the rate at which the population increases 
(respectively decreases) by one individual of the $j$th species is $\scaleto{K}{6pt}B_j(\scaleto{\frac{\vecn}{K}}{11pt})$ (respectively $\scaleto{K}{6pt}
D_j(\scaleto{\frac{\vecn}{K}}{11pt})$), where $K$ is a scaling parameter.
Under the assumptions we will make, the process goes extinct, \emph{i.e.}, $\veczer$ is an absorbing state, with probability one.
There are two limiting regimes for the behavior of this process.
The first one is to fix $K$ and let $t$ tend to infinity, which leads inevitably to extinction.
The second one consists in fixing a time horizon and letting $K$ tend to $+\infty$, after having rescaled the process
by $K$. In this limit, the behavior of the rescaled process is governed by a certain differential equation. 
More precisely, given any $0<t_{\scaleto{H}{4pt}}<+\infty$ and any $\varepsilon>0$ and $x_{0}\in \rpd$, we have
\[
\lim_{K\to+\infty} \proba_{\lfloor \sK x_{0}\rfloor}
\left(\,\sup_{0\,\leq t\leq\, t_{\scaleto{H}{3pt}}}\textup{dist}\left( \frac{\leproc(t)}{K},\vecx(t)\right)>\varepsilon\right)=0
\]
where $\textup{dist}(\cdot,\cdot)$ is the Euclidean distance in $\rpd$, and $\vecx(t)$ is the solution of the differential equation in $\rpd$
\[
\frac{\dd\vecx}{\dd t}=\vecB(\vecx)-\vecD(\vecx)
\]
with initial condition $\vecx_0$. We refer to \cite[Chapter 11]{EK} for a proof.
We use the notations $\vecx=(x_{\scaleto{1}{4pt}},\ldots,x_{\scaleto{d}{4pt}})$, $\vecB(\vecx)=(B_{\scaleto{1}{4pt}}(\vecx),\ldots,B_{\scaleto{d}{4pt}}(\vecx))$, 
and so on and so forth. We will make further assumptions (see Subsection \ref{hypo}) on the birth and death rates to be in the following 
situation. 
The vector field 
\[
\vecX=\vecB-\vecD
\]
has a unique attracting fixed point $\vecxf$ (lying in the interior of $\rpd$). We denote by $M^*$ its differential  evaluated at  
$\vecxf$, namely
\[
M^*=D\vecX(\vecxf).
\]
We then define the (engineering) resilience as
\[
\res= - \sup\{ \mathrm{Re} (z): z\in\Sp(M^*)\}
\]
where $\mathrm{Sp}\big(M^*\big)$ denotes the spectrum (set of eigenvalues) of the matrix $M^*$.
Under our assumptions, we have $\res>0$.

We can now formulate more precisely the goal of this paper. 
Given a \emph{finite-length} realization of the process $(\leproc(t), t\leq T)$, with large, but \emph{finite} $K$, we want to build an 
estimator for $\res$.
To this end, we need a good understanding of the behavior of the Markov process $(\leproc(t))$ in an intermediate regime between the 
two limiting regimes described above. This was done in a previous work of ours \cite{ccm2}, and this can be roughly summarized as follows.
All states $\vecn\neq \veczer$ are transient and $\veczer$ is absorbing, hence the only stationary distribution is the Dirac measure sitting at
$\veczer$. The mean time to extinction behaves like $\exp(\Theta(K))$. (We recall Bachmann-Landau notations below.)
If we start in the vicinity of the state $\vecnf = \lfloor K\vecxf\rfloor$, that is, if the initial state has its coordinates of size of order $K$,
then either the process wanders around $\vecnf$ or it gets absorbed at $\veczer$. 
More precisely, there is a unique quasi-stationary distribution (qsd, for short) $\nuk$ which describes the statistics of the process conditioned 
to be non-extinct before time $t$. Without this conditionning, the law of the process at time $t$ is well approximated by a mixture of the 
Dirac measure at $\veczer$ and the qsd $\nuk$, for times $t\in\left[ c K\log K, \exp(\Theta(K))\right]$, where $c>0$, in the sense that the total variation 
distance between them is exponentially small in $K$, provided that $K$ is large enough.
We will rely on these results that will be recalled precisely later on. We will also need to prove further properties.

\subsection{Main results}\label{subsec:main-results}

To estimate the engineering resilience $\res$, we will establish a matrix relation involving $M^*$. 
Let $\vmu^{\sK}=(\mu_1^{\sK},\ldots,\mu_d^{\sK})$ be the vector of species sizes averaged with respect to $\nuk$, that is,
\begin{equation}
\label{def-mup}
\mu_{p}^{\sK}=\int n_{p}\,\dd\nuk(\vecn)\,,\; p=1,\ldots,d.
\end{equation}
For each $\tau\geq 0$, define the matrix 
\[
\Sigma^{\sK}_{p,q}(\tau)=
\E_{\snuk}\!\left[\big(N_{p}^{\sK}(\tau)-\mu_{p}^{\sK}\big)\big(N_{q}^{\sK}(0)-\mu_{q}^{\sK}\big)\right],\;p,q \in \{1, \ldots, d\}.
\]
In Section \ref{proof:onsager}, we will prove the following result.
\begin{theorem}
\label{thm-onsager-relation}
For all $\tau\geq 0$ we have
\begin{equation}\label{onsager}
\Sigma^{\sK}(\tau)=\e^{\tau M^*}\Sigma^{\sK}(0)+\mathcal{O}\big(\sqrt{K}\,\big).
\end{equation}
\end{theorem}
Some comments are in order. If $\tau$ is equal to, say, $1/K$ then the estimate becomes useless. More generally, if $\tau$ is too small then
$\e^{\tau M^*}$ is too close to the identity matrix.  
Moreover, we will show later on that $\vmu^{\sK}$ and $\Sigma^{\sK}(\tau)$ are of order $K$. Hence the estimate becomes irrelevant 
if $\tau$ becomes proportional to $\log K$. Indeed, without knowing the constant of proportionality, $\e^{\tau M^*}\Sigma^{\sK}(0)$ can be of the same 
order than the error term.

Before proceeding further, we recall the following classical Bachmann-Landau notations.
\begin{notations}
Given $a\in\real$, the symbol $\mathcal{O}(K^a)$ stands for any real-valued function $f(K)$ such that 
there exists $C>0$ and $K_0>0$ such that, for any $K>K_0$, $|f(K)|\leq C K^a$. Note in particular that $\Oun$ will always mean a strictly positiveconstant that we don't want to specify.
Sometimes, we will also use the symbol $\Theta(K^a)$ stands for any real-valued function $f(K)$ such that there exist $C_1,C_2>0$ and $K_0>0$ 
such that, for any $K>K_0$, $C_1 K^a\leq f(K)\leq C_2 K^a$. One can naturally generalize $\Theta(K^a)$ to vector-valued functions. For instance, for $
\vecn\in\rpd$ we write $\vecn=\Theta(K^a)$ if $n_i=\Theta(K^a)$ for $i=1,\ldots,d$.
\end{notations}
Relation \eqref{onsager} allows to determine $M^*$. Indeed, we have
\begin{equation}\label{matrix-onsager}
\e^{\tau M^*} = \Sigma^{\sK}(\tau)\,\Sigma^{\sK}(0)^{-1} + \mathcal{O}\left(\frac{1}{\sqrt{K}}\right).
\end{equation}
This formula suggests that in order to estimate $M^*$, we need estimators for  $\Sigma^{\sK}(0)$ and  $\Sigma^{\sK}(\tau)$.
Given a finite-length realization of $\big(\leproc(t), 0\le t\le T\big)$ up to some time $T>0$, we define  estimators for 
$\mu_{p}^{\sK}$ and $\Sigma^{\sK}_{p,q}(\tau)$, 
for $0\le \tau<T$, $p,q\in\{1,\ldots, d\}$, $K\in \mathbb{N}^*$ by 
\begin{align}
\label{stat-mu}
& S^{{\scriptscriptstyle\vmu}}_{p}(T,K) =\frac{1}{T}\int_{0}^{T}N_{p}^{\sK}(s) \dd s
\end{align}
and
\begin{align}
& S^{{\scriptscriptstyle C}}_{p,q}(T,\tau,K)  =\nonumber\\
\label{stat-sigma}
& \qquad \frac{1}{T\!-\!\tau}\int_{0}^{T-\tau}\!\!\big(N_{p}^{\sK}(s+\tau)\!-\!S^{{\scriptscriptstyle\vmu}}_{p}(T,K)\big)
\big(N_{q}^{\sK}(s)\!-\!S^{{\scriptscriptstyle\vmu}}_{q}(T,K)\big)\dd s.
\end{align}
Under suitable conditions on $\vecn$, $K$ and $T$, $S^{{\scriptscriptstyle\vmu}}(T,K)$ well approximates $\vmu^{\sK}$.
More precisely, we will prove an estimate of the following form (see Theorem \ref{pseudoergod} for a precise statement)
\begin{align}
\nonumber
& \left|\E_{\vecn}\big[S^{{\scriptscriptstyle\vmu}}_{p}(T,K)\big]-\mu_p^{\sK}\right| \leq \\
\label{control-est-mu}
& \qquad C\big(K+\norm{\vecn}\big)\left(\frac{1+\log K}{T} + \e^{-c\,(\scaleto{\|}{7pt}\vecn\scaleto{\|}{7pt}_{\scaleto{1}{3pt}}\wedge K)}+\, T\e^{-\,c' K} \right)
\end{align}
for every $\vecn\in\domaine$, $p=1,\ldots,d$, where $c,c',C$ are positive constants. We use the notation
$\norm{\vecn}=\sum_{i=1}^d n_i$.
Let us comment on this bound. 
Roughly speaking, it can only be useful if $T$ is much smaller than $\exp(\Oun K)$ if $\vecn$ is, say, of order $K$. 
For instance, suppose that, for $K$ large enough, we want the bound to be $\Theta(K^{-a})$, for some $a>0$. One can check that this is possible if
$\vecn=\Theta(K)$ and $T=\Theta(K^{a+1} \log K)$. (Note in particular that, in this situation, we have a consistent estimator when $K\to+\infty$.)
However, when $T$ becomes $\exp(\Oun K)$ or larger, we know that $\E_{\vecn}\big[S^{{\scriptscriptstyle\vmu}}_{p}(T,K)\big]\approx 0$, because with high 
probability, at this time scale the process is absorbed at $\veczer$. This is the manifestation of the fact that the only stationary distribution is the Dirac 
measure at $\veczer$. Consistently, our bound becomes very bad in that regime.\newline
An estimate of the same kind holds for $S^{{\scriptscriptstyle C}}(T,\tau,K)$ which well approximates $\Sigma^{\sK}(\tau)$ in the appropriate ranges.

\begin{remark} It is possible to use discrete time  instead of continuous time in the above averages. Indeed the key results (in particular Proposition 
\ref{pre-theorem-ergodic}) are obtained for discrete times.
\end{remark} 

We can now define the empirical matrix $M^{*}_{\scriptscriptstyle{\mathrm{emp}}}(T,\tau,K)$  by
\[
\e^{\tau M^{*}_{\scriptscriptstyle{\mathrm{emp}}}(T,\tau,K)}=S^{{\scriptscriptstyle C}}(T,\tau,K)\,S^{{\scriptscriptstyle C}}(T,0,K)^{-1}.
\]
We will see later on that, in appropriate regimes,  $S^{{\scriptscriptstyle C}}(T,0,K)$ is near $\Sigma^K(0)$ and $S^{{\scriptscriptstyle C}}(T,\tau,K)$ is near $\Sigma^K(\tau)$ 
(see Propositions \ref{vs2} and \ref{vs4}). The matrix $\Sigma^K(0)$ is invertible as a covariance matrix of a non-constant vector and is $\Theta(K)$ (see Proposition 
\ref{borninf}). Then  \eqref{onsager} implies that $\Sigma^K(\tau)$ is invertible and the same holds for $S^{{\scriptscriptstyle C}}(T,\tau,K)$. These remarks imply 
that the matrix  $M^{*}_{\scriptscriptstyle{\mathrm{emp}}}$ is well defined. 

We define  the empirical resilience by  
\[
\rho^{\scaleto{*}{3pt}}_{\scriptscriptstyle{\mathrm{emp}}}(T,\tau,K)=
- \sup\big\{ \operatorname{Re} (z): z\in\mathrm{Sp}\big(M^{*}_{\scriptscriptstyle{\mathrm{emp}}}(T,\tau,K)\big)\big\}.
\]
Our main result (Theorem \ref{rhoemp}) is then the following. 
\begin{theorem*}
For $\tau=\Theta(1)$, $\vecn=\Theta(K)$ (initial state) and $0<T\ll \exp(\Thetaun K)$,
and $K$ large enough, we have
\[
\big|\rho^{\scaleto{*}{2.8pt}}_{\scriptscriptstyle{\mathrm{emp}}}(T,\tau,K)-\res\big|\le
\Oun \left( \frac{K^2}{\sqrt{T}} +\frac{1}{\sqrt{K}}\right)
\]
with a probability larger than $1-1/K$. In particular, if $T\geq K^5$, we have
\[
\big|\rho^{\scaleto{*}{2.8pt}}_{\scriptscriptstyle{\mathrm{emp}}}(T,\tau,K)-\res\big|\leq \frac{\Oun}{\sqrt{K}}.
\]
\end{theorem*}
Several comments are in order.
The dependence on the initial state $\vecn$ is somewhat hidden and involved in the fact that the estimates hold ``with a probability larger than $1-1/K$''.
Indeed, the estimate of this probability results from Chebychev inequality and variance estimates in which the process is started in
$\vecn$.
What the symbol $\ll$ precisely means is not mathematically defined. It means that we need to consider $T$ ``much smaller than 
something exponentially big in $K$''. Indeed, since we do not control explicitly the various constants appearing in exponential terms in $K$, we have to 
consider $T$ which varies on a scale smaller than $\exp(\Theta(1)K)$, for instance $\exp\big(\Theta(1)\sqrt{K}\,\big)$. The 
reader is invited to step through the proof of Theorem \ref{rhoemp} for the more precise, but cumbersome bound we obtain.

\subsection{A ``fluctuation-dissipation'' approach}

The above estimator for the engineering resilience, based on \eqref{matrix-onsager}, is valid for any $d$. 
In the case $d=1$ (only one species), we  have another, simpler, estimator based on a ``fluctuation-dissipation relation''.
This relation is in fact true for any $d$ and of independent interest. 
Let $\EuScript{D}^{\sK}$ be the $d\times d$ diagonal matrix given by
\[
\EuScript{D}^{\sK}_{p,p}=K B_{p}(\vecxf)=K D_{p}(\vecxf).
\]
We have the following result. We write $\Sigma^{\sK}$ instead of $\Sigma^{\sK}(0)$, and the transpose of a matrix $M$ is denoted
by $M^{\intercal}$.
\begin{theorem}\label{thm-Kubo}
We have
\begin{equation}\label{kuboeq}
M^*\Sigma^{\sK}+\Sigma^{\sK}{M^*}^{\intercal}+2\EuScript{D}^{\sK}=\mathcal{O}\big(\sqrt{K}\,\big).
\end{equation}
\end{theorem}
This relation is proved in Section \ref{proof:kuboeq}. For background on fluctuation-dissipation relations in Statistical 
Physics, we refer to \cite[sections 2-3]{kubo}. Note that the matrix $\Sigma^{\sK}$ is symmetrical, but in general the matrix 
$M^*$ is not (see 
\cite{ccm2}). Note also that each term in the left-hand side of \eqref{kuboeq} is of order $K$, as we will see below.

If $\Sigma^{\sK}$ and $\EuScript{D}^{\sK}$ are known, the matrix $M^*$ is not  uniquely defined, except 
for $d=1$ (see  for example \cite{sim}).
For $d=1$, \eqref{kuboeq} easily gives the resilience since it becomes a scalar equation: 
\[
\res = \frac{K(B(\vecxf)+D(\vecxf))}{2 \Sigma^{\sK}} +\mathcal{O}\left(\frac{1}{\sqrt{K}}\right).
\]
\begin{remark}
The quantity $K(B(\vecxf)+D(\vecxf)) = 2K B(\vecxf)$ is the average total jump rate $K \nuk(B(\vecn/K) +D(\vecn/K))$ up to $\Oun$ terms.
This follows from a Taylor expansion of $B(\vecn/K) +D(\vecn/K)$ around $\vecxf$, Theorem \ref{bornemom} and Proposition \ref{moyenne}.
\end{remark}
In the case $d=1$, an estimator for $\EuScript{D}^{\sK}$  is
\begin{equation}
\label{stat-D}
S^{{\scriptscriptstyle\EuScript{D}}}(T,K)=\frac{1}{T}\big(\mathrm{number\; of \; births\;up\; to\; time}\;T\big).
\end{equation}
In Section \ref{sec:stat}, we establish a bound for
\[
\left|\E_{\vecn}\big[S^{{\scriptscriptstyle\EuScript{D}}}(T,K)\big]- K B(\vecxf)\right|
\]
which depends on $T, K$ and $\norm{\vecn}$, and is small in the relevant regimes. 
The estimator we use for $\Sigma^{\sK}$ is
\begin{equation}
\label{stat-MSigma}
S^{{\scriptscriptstyle\Sigma}}(T,K)=
\frac{1}{T}\int_{0}^{T}\big(N^{\sK}(s)-S^{{\scriptscriptstyle\vmu}}(T,K)\big)\big(N^{\sK}(s)-S^{{\scriptscriptstyle\vmu}}(T,K)\big)\dd s.
\end{equation}
Again, we can control how well this estimator approximates $\Sigma^{\sK}$. This provides another estimator for $\rho^{\scaleto{*}{3pt}}$, with a controlled error.  

\subsection{Standing assumptions}\label{hypo}

The two (regular) vector fields $\vecB(\vecx)$ and $\vecD(\vecx)$ are given in $\rpd$. 
We assume that their components have second partial derivatives which are polynomially bounded.
Obviously, we suppose that $B_j(\vecx)\geq 0$ and $D_j(\vecx)\geq 0$ for all $j=1,\ldots,d$ and $\vecx\in\rpd$.
A dynamical system in $\rpd$ is defined by the vector field $\vecX(\vecx)=\vecB(\vecx)-\vecD(\vecx)$, namely
\[
\frac{\dd \vecx}{\dd t}=\vecB(\vecx)-\vecD(\vecx)=\vecX(\vecx).
\]
For $\vecx\in \rpd$, we use the following standard norms:
\[
\norm{\vecx}=\sum_{j=\scaleto{1}{4.5pt}}^{\scaleto{d}{5pt}}x_{j}\,, \; 
\normdeux{\vecx}=\sqrt{\sum_{j=\scaleto{1}{4.5pt}}^{\scaleto{d}{5pt}} x_{j}^{\scaleto{2}{4pt}}}.
\]
We now state our hypotheses.

\begin{enumerate}[label=\textbf{H.\arabic*}]
\item \label{H1} The vectors fields $\vecB$ and $\vecD$ vanish only at $\ushort{0}$. 
\item \label{H2} There exists $\vecxf$ belonging to the interior of $\rpd$ \textup{(}fixed point of $\vecX$
\textup{)} such that
\[
\vecB(\vecxf)-\vecD(\vecxf)=\vecX(\vecxf)=\ushort{0}\,.
\]
\item\label{H3} Attracting fixed point: there exist $\beta>0$ and $R>0$ such that $\normdeux{\vecxf}<R$, and for all $\vecx\in \rpd$ with $\normdeux{\vecx}<R$,
\begin{equation}\label{Xalpha}
\langle\, \vecX(\vecx), (\vecx-\vecxf)\rangle \le -\beta\, \normdeux{\vecx}\, \scaleto{\|}{10pt}\vecx-\vecxf\scaleto{\|}{10pt}_{\scaleto{2}{4pt}}^{\scaleto{2}{4pt}}\,.
\end{equation}
\item\label{H4} The fixed point $\ushort{0}$ of the vector field $\vecX$ is repelling (locally unstable).
Moreover, on the boundary of $\rpd$, the vector field $\vecX$ points toward the interior (except at $\veczer$).
\item\label{H5} Define
\[
\widehat{B}(y)=\sup_{\snorm{\vecx}=y}\sum_{j=\scaleto{1}{4.5pt}}^{d}B_{j}(\vecx)\,,\;  \widehat{D}(y)=
\inf_{\snorm{\vecx}=y}\sum_{j=\scaleto{1}{4.5pt}}^{d}D_{j}(\vecx)
\]
and for $y>0$, let
\[
F(y)= \frac{\widehat{B}(y)}{\widehat{D}(y)}.
\]
We assume that there exists $\,0<L<R\,$ such that $\sup_{y>L} F(y)<\scaleto{1/2}{10pt}$ and $\lim_{y\to+\infty} F(y)=0.$
\item\label{H6} There exists $y_{0}>0$ such that $\int_{y_{0}}^{\infty} \widehat{D}(y)^{-1}\dd y<+\infty$ and $y\mapsto \widehat{D}(y)$ is 
increasing on $\left[y_{0},+\infty\right[$.
\item 
There exists $\xi>0$ such that
\begin{equation}
\label{cond:mort}
\tag{H7}
\inf_{\vecx\,\in \srpd} \inf_{1\leq\, j\,\leq\, d}\, \frac{D_{j}(\vecx)}{\sup_{1\leq\, \ell\, \leq \,d} x_{\ell}} >\xi.
\end{equation}
\item
Finally, we assume that
\begin{equation}\label{cond:nais}\tag{H8}
\inf_{1\leq\, j\,\leq \, d}\partial_{x_{j}} B_{j}(\veczer) > 0.
\end{equation}
(By $\partial_{x_{j}}$ we mean the partial derivative with respect to $x_{j}$.)
\end{enumerate}

Assumptions \ref{H5} and \ref{H6} ensure that the time for ``coming down from infinity'' for the 
dynamical system is finite. Together with \ref{H3}, this also implies that $\vecxf$ is a 
globally attracting stable fixed point. More comments on these assumptions can be found in
\cite{ccm2}. 

\subsection{A numerical example}

We consider the two-dimensional vector fields
\[
\vecB(x_{\scaleto{1}{4pt}},x_{\scaleto{2}{4pt}})=
\begin{pmatrix}
a\,x_{\scaleto{1}{4pt}}+b\,x_{\scaleto{2}{4pt}}\\
e\,x_{\scaleto{1}{4pt}}+f\,x_{\scaleto{2}{4pt}}
\end{pmatrix}
\]
and
\[
\vecD(x_{\scaleto{1}{4pt}},x_{\scaleto{2}{4pt}})=
\begin{pmatrix}
x_{\scaleto{1}{4pt}}\big(c\,x_{\scaleto{1}{4pt}}+d\,x_{\scaleto{2}{4pt}}\big)\\
x_{\scaleto{2}{4pt}}\big(g\,x_{\scaleto{1}{4pt}}+h\,x_{\scaleto{2}{4pt}}\big)
\end{pmatrix}
\]
where all the coefficients are positive. This is a model of competition between two species of Lotka-Volterra type.
We have taken
\begin{align*}
& a= 0.4569,\; b= 0.2959,\;e= 0.5920,\;f= 0.6449\\
& c= 0.9263,\;d= 0.9157,\;g= 0.9971,\;h= 0.2905.
\end{align*}
Assumptions \ref{H1} and \ref{H4} are easily verified numerically.
Assumptions \ref{H5} and \ref{H6} are satisfied because $\widehat{B}(y)\leq (a+b+e+f)y$ and
$\widehat{D}(y)\geq (c\wedge h)y^{\scaleto{2}{4pt}}/4$.
Concerning \ref{H2}, we checked numerically that there is a unique fixed point inside the positive quadrant, namely
$\vecxf=(0.3567,1.4855)$.
It remains to check \ref{H3}, namely that 
\[
-\beta= \sup\{R(\vecx):\vecx \in \real_{\scaleto{+}{5pt}}^{\scaleto{2}{4pt}}\}<0
\]
where
\[
R(\vecx)=\frac{\langle \vecX(\vecx),(\vecx-\vecxf)\rangle}{\normdeux{\vecx}\normdeux{\vecx-\vecxf}^{\scaleto{2}{3.6pt}}}.
\]
We first checked that the numerator $N(\vecx)=\langle\vecX(\vecx),(\vecx-\vecxf)\rangle$ is negative and vanishes only at 
$\veczer$ and $\vecxf$. It is easy to check that $N(\vecx)<0$ for $\normdeux{\vecx}$ large enough. We have verified
numerically that the only solutions of the equations $\partial_{x_1}N=\partial_{x_2}N=0$ in the closed positive quadrant are
$\vecxf$ and $\vecz=(0.1739, 0.4361)$, with $N(\vecz)= -0.2852$, thus this is a negative local minimum.
This implies that $N(\vecx)<0$ in the closed positive quadrant, except at $\veczer$ and $\vecxf$ 
where it vanishes. This implies that $R\leq 0$ in the closed positive quadrant.
It is easy to check that
\[
\limsup_{\snormdeux{\vecx}\to\scaleto{+\infty}{5pt}} R(\vecx)\leq -(c\wedge h)/\scaleto{\sqrt{2}}{9pt}.
\]
This implies that $R<0$ except perhaps at $\veczer$ and $\vecxf$.
Near $\veczer$ we have by Taylor expansion
\[
R(\vecx)
=-\frac{\langle D\vecX(0)\,\vecx,\vecxf\rangle}{\normdeux{\vecx}\normdeux{\vecxf}^{{\scaleto{2}{3.6pt}}}}\,
\big(1+\mathcal{O}(\normdeux{\vecx})\big)=
-\frac{\langle\, \vecx,D^{\intercal}\vecB(0)\,\vecxf
\rangle}{\normdeux{\vecx}\normdeux{\vecxf}^{{\scaleto{2}{3.6pt}}}}\,
\big(1+\mathcal{O}(\normdeux{\vecx})\big)
\]
and, since the vector $D^{\intercal}\vecB(0)\,\vecxf$ has positive components, there
exists $\varrho>0$ such that for all $\vecx\in\real_{\scaleto{+}{5pt}}^{\hspace{-.2pt}\scaleto{2}{4pt}}$
\[
\langle\, \vecx,D^{\intercal}\vecB(0)\,\vecxf\rangle\ge \varrho\,\normdeux{\vecx}.
\]
If $\vecy=\vecx-\vecxf$ is small, we have by Taylor expansion (since $\vecX(\vecxf)=\veczer$)
\[
R(\vecx)
=\frac{\langle M^*\vecy,\vecy\rangle}{\normdeux{\vecxf}\,\normdeux{\vecy}^{\scaleto{2}{3.6pt}}}\,
\big(1+\mathcal{O}(\normdeux{\vecy})\big)
=\frac{\left\langle\,\vecy,\frac{1}{2}\,\big({M^*}^{\intercal}+M^*\big)\,\vecy\big)
\right\rangle}{\normdeux{\vecxf}\,\normdeux{\vecy}^{\scaleto{2}{3.6pt}}}\,
\big(1+\mathcal{O}(\normdeux{\vecy})\big).
\]
One can check numerically that the two real eigenvalues of the symmetric matrix
\[
{M^*}^{\intercal}+M^*
\]
are strictly negative, the largest being numerically equal to $-0.786$. This completes the verification of hypothesis
\ref{H3}.

Illustrating standard experiments on populations of cells or bacteria, we have chosen $K=10^5$ and simulated a unique realization of the process with $T=100$ which contains 
about $5.10^7$ jumps (cell divisions or deaths).
The resilience computed from the vector field is numerically equal to $0.547$. We have computed $\rho^{\scaleto{*}{3pt}}_{\scriptscriptstyle{\mathrm{emp}}}(100,1,10^5)$.
The relative error, that is $|\rho^{\scaleto{*}{3pt}}_{\scriptscriptstyle{\mathrm{emp}}}(100,1,10^5)-\rho^{\scaleto{*}{3pt}}|/\rho^{\scaleto{*}{3pt}}$, is equal to $0.022$.

\smallskip \noindent
Note that the situation we are interested in is completely different from standard statistical approach where one can repeat the experiments.


\subsection{Organization of the paper}

In Section \ref{moment-process}, we will study the time evolution  of the moments of the process and  we will prove moment 
estimates for the qsd.
In Section \ref{sec:process}, we will obtain control on the large time behavior of averages for the 
process. In Section \ref{chaleur}, we will prove the relations \eqref{onsager}  and \eqref{kuboeq}.
In Section \ref{sec:stat}, we will apply these relations to obtain approximate expressions of the engineering resilience in terms of the covariance 
matrices for the qsd. 
From the results of Section \ref{sec:process}, we will deduce variance bounds for the estimators \eqref{stat-mu}, \eqref{stat-sigma} 
and \eqref{stat-D}, starting either in the qsd or from an initial condition of order $K$. 

\section{Time evolution of moments of the process and moments of the QSD}
\label{moment-process}

\subsection{Time evolution of moments starting from anywhere}
The generator $\gen$ of the birth and death process $\leproc=(\leproc(t),t\geq 0)$ is defined by
\begin{align}\label{gene}
& \gen f(\vecn)=\\
& K\sum_{\ell=1}^{d}B_{\ell}\left(\frac{\vecn}{K}\right)\big(f(\vecn+\vece{\ell})-f(\vecn)\big)
+K\sum_{\ell=1}^{d}D_{\ell}\left(\frac{\vecn}{K}\right)\big(f(\vecn-\vece{\ell})-f(\vecn)\big)
\nonumber
\end{align}
where $\vece{\ell}=(0,\ldots,0,1,0,\ldots,0)$, the $1$ being
at the $\ell$-th position, and $f:\domaine\to\real$ is a function with bounded support.
We denote by $(S^{\sK}_{t},t\geq 0)$ the semigroup of the process $\leproc$ acting on bounded 
functions, that is, for $f:\domaine\to\real$, we have
\[
S^{\sK}_{t}f(\vecn)
=\E\!\left[ f(\leproc(t))\big|\, \leproc(0)=\vecn\right]
=\E_{\vecn}\!\left[ f(\leproc(t))\right].
\]
For $A>1$, let
\begin{equation}\label{leTA}
\mathcal{T}_A=\inf\{t>0:\|\leproc(t)\|_{\scaleto{1}{4pt}}>A\}.
\end{equation}
Notice that we will use either $\norm{\!\hspace{-0.5pt}\cdot\!}$ or $\normdeux{\!\hspace{-0.5pt}\cdot\!}$. They are of course equivalent but one can be more convenient than the other, depending on the context.
We have the following result.

\begin{theorem}\label{borneexp}
There exists a constant $\cte{borneexp}>0$ such that for $K$ large enough, the operator
group $S^{\sK}_{1}$ extends to exponentially bounded functions and 
\[
\sup_{\vecn\, \in\, \sdomaine} S^{\sK}_{1}\left(\e^{\scaleto{\|}{7pt}\hspace{-0.1pt}\cdot\scaleto{\|}{7pt}_{\scaleto{1}{3pt}}}\right)(\vecn)\le 
\e^{\cte{borneexp} K}.
\]
\end{theorem}
\begin{proof}
Introduce the function $G_{\sK}$ defined on $[\hspace{1pt}y_{0},+\infty)$ by
\[
G_{\!\sK}(y) = \int_{y}^{\infty}\frac{\dd z}{\widehat{D}(z)}+\frac{1}{K\,\widehat{D}(y)}.
\]
Assumption \ref{H6} implies that  $G_{\!\sK}$ is well defined and decreasing on $[y_{0},+\infty)$. We can define its 
inverse function on $(0,s_{0}]$ for $s_{0}>0$ small enough (independent of $K$). 
Take $0<\eta \leq s_{0}\wedge \frac{1-\,\e^{-1}}{4}$.  Then there is a unique positive function $y_{\sK}$ defined by
\begin{equation}
\label{yK}
y_{\sK}(s)=G_{\!\sK}^{-1}(\eta s),\;s\in(0,1]. 
\end{equation}
Note that $y_{\sK}(s)\geq y_{0}$ and $\lim_{s\downarrow 0}y_{\sK}(s)=+\infty$.
Let
\[
\varphi_{\sK}(s)=\frac{\e^{-K y_{\scaleto{K}{3pt}}(s)}}{K \widehat{D}(y_{\sK}(s))}.
\]
Note that
\[
\lim_{s\hspace{.2pt}\downarrow\hspace{.2pt}0}\varphi_{\sK}(s)=0.
\]
Using the  Lipschitz continuity of $\widehat{D}$ (and then its differentiability almost everywhere) and \eqref{yK}, we obtain
\[
\dot{\varphi}_{\sK}(s)=\frac{\dd\varphi_{\sK}}{\dd s}(s)=
-\!\left(\!\frac{\e^{-K y_{\scaleto{K}{3pt}}(s)}}{\widehat{D}(y_{\sK}(s))}\!+
\!\frac{\e^{-K y_{\scaleto{K}{3pt}}(s)} \!\widehat{D}'(y_{\sK}(s))}{K \widehat{D}(y_{\sK}(s))^{\scaleto{2}{4.5pt}}}\!\right) \!\!\frac{\dd y_{\sK}}{\dd s}(s)
= \eta \e^{-K y_{\scaleto{K}{3pt}}(s)}\!.
\]
We now consider the function
\[
f\!_{\sK}(t,\vecn)=\varphi_{\sK}(t)\e^{\scaleto{\|}{7pt}\vecn\scaleto{\|}{7pt}_{\scaleto{1}{3pt}}}
\]
to which we apply It\^o's formula at time $t\wedge \mathcal{T}_A$.  We get
\[
\E_{\vecn}\!\left[\varphi_{\sK}\big(t\wedge \mathcal{T}_A\big)\e^{\|\sleproc(t\wedge \mathcal{T}_A)\|_{\scaleto{1}{3pt}}}\right]=
\E_{\vecn}\!\left[\, \int_{0}^{t\wedge \mathcal{T}_A}\big(\partial_{t}f\!_{\sK}+\gen f\!_{\sK}\big)(s,\leproc(s))\dd s\right].
\]
We have 
\begin{align*}
\MoveEqLeft \partial_{t}f\!_{\sK}(t,\vecn)+\gen f\!_{\sK}(t,\vecn)=
\dot\varphi_{\sK}(t)\e^{\scaleto{\|}{7pt}\vecn\scaleto{\|}{7pt}_{\scaleto{1}{3pt}}}\\
& +K \varphi_{\sK}(t)\e^{\scaleto{\|}{7pt}\vecn\scaleto{\|}{7pt}_{\scaleto{1}{3pt}}}\left((\e-1)\sum_{\ell=1}^{d}B_{\ell}\left(\frac{\vecn}{K}\right)
+(\e^{-1}-1)\sum_{\ell=1}^{d}D_{\ell}\left(\frac{\vecn}{K}\right)\right).
\end{align*}
Note that 
\begin{align*}
& \partial_{t}f\!_{\sK}(t,\vecn)+\gen f\!_{\sK}(t,\vecn)\\
& \leq \e^{\scaleto{\|}{7pt}\vecn\scaleto{\|}{7pt}_{\scaleto{1}{3pt}}}\left(\dot\varphi_{\sK}(t)
+ K\varphi_{\sK}(t) \left((\e-1)\,\widehat{B}\left(\frac{\norm{\vecn}}{K}\right)-(1-\e^{-1})\,\widehat{D}\left(\frac{\norm{\vecn}}{K}\right)\right)\right)
\\
&\leq 
\e^{\scaleto{\|}{7pt}\vecn\scaleto{\|}{7pt}_{\scaleto{1}{3pt}}}\left(\dot\varphi_{\sK}(t)- K\varphi_{\sK}(t)(1-\e^{-1})\,\widehat{D}\left(\frac{\norm{\vecn}}{K}\right)
\left(1 - \e F\!\left(\frac{\norm{\vecn}}{K}\right)\right)\right).
\end{align*}
It follows from \ref{H5} that there exists  a number $\zeta>y_{0}$ such that if $y>\zeta$, then $F(y)<(2 \hspace{-1pt}\e)^{-1}$. \newline
If $\norm{\vecn}<\zeta K$ we get
\[
\big|\partial_{t}f\!_{\sK}(t,\vecn)+\gen f\!_{\sK}(t,\vecn)\big|\le 
\Oun \e^{\zeta K}\big(\dot\varphi_{\sK}(t)+K\varphi_{\sK}(t)\big).
\]
For $\norm{\vecn}\ge K(\zeta\vee y_{\sK}(t))$ we have
\[
\partial_{t}f\!_{\sK}(t,\vecn)+\gen f\!_{\sK}(t,\vecn)\le 0
\]
since $\dot\varphi_{\sK}(t) = \eta K \widehat{D}(y_{\sK}(t))\varphi_{\sK}(t)$ and $\widehat{D}(\norm{\vecn}/K)\geq \widehat{D}(y_{\sK}(t))$. \newline
Finally,  for $\zeta K\le \norm{\vecn}<K y_{\sK}(t)$ we get
\[
\big|\partial_{t}f\!_{\sK}(t,\vecn)+\gen f\!_{\sK}(t,\vecn)\big|\le \e^{K y_{\sK}(t)}\dot\varphi_{\sK}(t) = \eta.
\]
We deduce that
\[
\E_{\vecn}\left[\varphi_{\sK}\big(1\wedge \mathcal{T}_A\big)\e^{\|\sleproc(1\wedge \mathcal{T}_A)\|_{\scaleto{1}
{3.5pt}}}\right]
\le \Oun \e^{\zeta K}.
\]
The result follows by letting $A$ tend to infinity and by monotonicity.
\end{proof}

We deduce moment estimates for the process which are uniform in the starting state, and in time, for times larger than $1$.
\begin{corollary}\label{auxbornees}
For all $t\ge1$, the semi-group $(S_{t})$ maps functions of polynomially bounded modulus in  bounded 
functions. 
In particular, for all $q\in\integers$, we have 
\begin{equation}
\label{estim-chili}
\sup_{t\geq 1}\sup_{\vecn\,\in\, \sdomaine} \E_{\vecn} \big[ \|\leproc(t)\|_{\scaleto{1}{4pt}}^{\scaleto{q}{4pt}}\big]
\leq q^q \e^{-q} K^q \e^{\cte{borneexp}}.
\end{equation}
\end{corollary}
\begin{proof}
We have
\begin{align*}
\E_{\vecn} \big[ \|\leproc(1)\|_{\scaleto{1}{4pt}}^{\scaleto{q}{4pt}}\big]
& = K^q\, \E_{\vecn} \!\left[ \frac{\|\leproc(1)\|_{\scaleto{1}{4pt}}^{\scaleto{q}{4pt}}}{K^q} \e^{-\frac{\|\sleproc(1)\|_{\scaleto{1}{3pt}}}{K}} \e^{\frac{\|\sleproc(1)\|_{\scaleto{1}{3pt}}}{K}}\right]\\
& \leq K^q q^q  \e^{-q} \E_{\vecn} \Big[\e^{\frac{\|\sleproc(1)\|_{\scaleto{1}{3pt}}}{K}}\Big]
\end{align*}
since for all $x\geq 0$, $x^q \e^{-x}\leq q^q \e^{-q}$.
Inequality \eqref{estim-chili} follows from H\"older's inequality and Theorem \ref{borneexp}.
Let us now consider $t> 1$. From the Markov property and by using the previous inequality, we deduce that
\[
\E_{\vecn}\left[\big\|\leproc(t)\big\|_{\scaleto{1}{4pt}}^{\scaleto{q}{4pt}}\, \right]
= \E_{\vecn}\Big[\E_{\sleproc(t-1)}\left[\big\|\leproc(1)\big\|^{\scaleto{q}{4pt}}_{\scaleto{1}{4pt}}\,\right]\Big] 
 \le q^q \e^{-q} K^q  \e^{\cte{borneexp}}.
\]
The proof is finished.
\end{proof}

For times $t$ less than $1$, the moment estimates depends on the initial state. 
\begin{proposition}\label{tempscourts}
For each integer $q$, there exists a constant $c_{q}>0$ such that for all $K>1$, $t\geq 0$ and $\vecn\in \domaine$
\[
\E_{\vecn}\big[\big\|\leproc(t)\big\|^{\scaleto{q}{4pt}}_{\scaleto{2}{4pt}}\,\big]\le c_{q} K^{q}
+\normdeux{\vecn}^{\scaleto{q}{4pt}} \, \un_{\{t< 1\}}\,.
\]
\end{proposition}
\begin{proof} We have only to study the case $t< 1$, the other case being given in \eqref{estim-chili}.
We prove the result for $q$ even, namely $q=\scaleto{2}{6pt}q'$. The result for $q$ odd follows from Cauchy-Schwarz inequality. 
Letting
\[
f_{q'}(\vecn)=\normdeux{\vecn}^{\scaleto{2q'}{5.4pt}} 
\]
we have
\begin{align*}
\MoveEqLeft \gen f_{q'}(\vecn)=K\sum_{\ell=1}^{d}B_{\ell}\left(\frac{\vecn}{K}\right)\Big(\big(\normdeux{\vecn}^{\scaleto{2}{3.6pt}}+2\vecn_{\ell}+1\big)^{\scaleto{q'}{6.5pt}}-
\normdeux{\vecn}^{\scaleto{2q'}{5.4pt}}\big)\Big)\\
&  \quad\qquad +K\sum_{\ell=1}^{d}D_{\ell}\left(\frac{\vecn}{K}\right)\Big(\big(\normdeux{\vecn}^{\scaleto{2}{3.6pt}}-2\,\vecn_{\ell}+1\big)^{\scaleto{q'}{6.5pt}}-
\normdeux{\vecn}^{\scaleto{2q'}{5.4pt}}\big)\Big).
\end{align*}
Using \ref{H5} and the equivalence of the norms, we see that there exists a constant $c_{q'}>0$ such that if
$\normdeux{\vecn}> c_{q'} K$
\[
\gen f_{q'}(\vecn)<0.
\]
Moreover, we can take $c_{q'}$ large enough such that for all $\vecn$
\[
\gen f_{q'}(\vecn)\le c_{q'} K^{2q'}.
\]
Applying It\^o's formula to $f_{q'}$ we get as in the proof of Theorem
\ref{borneexp} 
\begin{align*}
\E_{\vecn}\left[\,\|\leproc(t\wedge \mathcal{T}_A)\|^{\scaleto{2q'}{6pt}}_{\scaleto{2}{4pt}}\,\right]
& \le \normdeux{\vecn}^{\scaleto{2q'}{5.3pt}}+\E_{\vecn}\left[\, \int_{0}^{t\wedge \mathcal{T}_A} c_{q'} K^{2q'}\dd s\right]\\
& \le \normdeux{\vecn}^{\scaleto{2q'}{5.3pt}}+t\, c_{q'} K^{2q'}.
\end{align*}
(Recall that $\mathcal{T}_A$ is defined in \eqref{leTA}.)
The result follows  by letting $A$ tend to infinity.
\end{proof}

\subsection{Moments estimates for the qsd}

Let us first recall (cf. \cite{ccm2})  that, under the assumptions of Section \ref{hypo}, there exists a unique qsd $
\nuk$ with support  $\domaine\backslash \{\veczer\}$. 
Further,  starting from the qsd, the extinction time is distributed according to an exponential law with parameter 
$\lambda_{0}(K)$ satisfying (Theorem 3.2 in \cite{ccm2})
\begin{equation}
\label{eq:lambda0}
\e^{-d_{1} K}\le\lambda_{0}(K)\le \e^{-d_{2} K}
\end{equation}
where $d_{1}>d_{2}>0$ are constants independent of $K$. 
Recall also that for all $t>0$,
\begin{equation}
\label{pte-qsd}
\proba_{\!\nuk}\big(\leproc(t)\in \cdot\,,T_{\veczer}>t\big)=
\e^{-\lambda_{0}(K)\hspace{.4pt}t}\nuk\big(\cdot)
\end{equation}
where
\[
T_{\veczer}=\inf\{t>0: \leproc(t)=\veczer\}.
\]
Finally, for all $f$ in the domain of the generator
\begin{equation}
\label{eq:miracle}
 \gen^{\dagger}\nuk(f)=\nuk(\gen f)= - \lambda_{0}(K)\,\nuk( f)
\end{equation}
with the notation
\[
\nuk(f)=\int f(\vecn) \dd\nuk(\vecn).
\]
We use several notations from \cite{ccm2} that we now recall. Let
\[
\vecnf= \lfloor K \vecxf\rfloor.
\] 
For $\vecx\in\rpd$ and $r>0$, $\mathcal{B}(\vecx,r)$ is the ball of center $\vecx$ and radius $r$.
We consider the sets
\begin{equation}
\label{boule}
\Delta = \mathcal{B}\big(\vecn^{\scaleto{*}{3pt}},\rho_{(4.2)}\sqrt{K}\,\big), \;\mathcal{D} =  \mathcal{B} \left(\vecnf,\frac{\min_{j} n_{j}^{\scaleto{*}{3pt}}}{2}\right)\cap \domaine
\end{equation}
where $\rho_{(4.2)}>0$ is a constant defined in \cite[Corollary 4.2]{ccm2}. 
Note that since $\vecnf$ is of order $K$, we have 
$\Delta \subset \mathcal{D}$ for $K$ large enough. The first entrance time in $\Delta$ (resp. $
\mathcal{D}$) will be denoted by $T_{\!\sDelta}$ (resp. $T_\mathcal{D}$). 

We first prove that the support of the qsd is, for large $K$, almost included in $ \mathcal{D}$. (This will be important to control moments later on.)
\begin{proposition}\label{borneexterieur}
There exists a constant $\pcte{borneexterieur}>0$ such that for all $K$ large enough
\[
\nuk\big( \mathcal{D}^{c}\big)\le \e^{-\pcte{borneexterieur}K}.
\]
\end{proposition}
\begin{proof}
We first recall two results from \cite{ccm2}. 
From Lemma 5.1  in \cite{ccm2}, there exist $\gamma>0$ and $\delta\in(0,1)$ such that for all $K$ large enough
\begin{equation}\label{res1}
\sup_{\vecn\in\sDelta^{\!c}\backslash\veczer}\proba_{\!\vecn}\big(T_{\sDelta}>\gamma\log K,T_{\veczer}>T_{\sDelta}\big)\leq \delta.
\end{equation}
By Sublemma 5.8 in \cite{ccm2}, there exist two constants $C>0$ and $c>0$ such that for all $K$ large enough, and for all $t>0$
\begin{equation}\label{res2}
\sup_{\vecn\in\Delta}\proba_{\!\vecn}\big(T_{ \mathcal{D}^{c}}<t\big)\leq C\big(1+t\big) \e^{-c K}.
\end{equation}
Now, for $q\in\integers\backslash \{0\}$ define
\[
t_{q}=q\gamma\log K.
\]
We will first estimate
$\sup_{\vecn}\proba_{\!\vecn}\big(\leproc(t_{q})\in  \mathcal{D}^{c}, T_{\veczer}>t_q\big)$.
Note that $\leproc(t_{q})\in  \mathcal{D}^{c}$ implies $T_{ \mathcal{D}^{c}}\leq t_{q}$. 
We distinguish two cases according to whether $\vecn\in \Delta$ or $\vecn\in \Delta^{\!c}\backslash\{\veczer\}$.\newline
Let $\vecn\in \Delta$. It follows from \eqref{res2} that
\[
\proba_{\!\vecn}\left(\leproc(t_q)\in  \mathcal{D}^{c}\right)\le C\big(1+t_{q}\big)\e^{-cK}. 
\]
Now let $\vecn\in \Delta^{\!c}\backslash\{\veczer\}$. We have
\begin{align*}
\MoveEqLeft \proba_{\!\vecn}\big(\leproc(t_q)\in  \mathcal{D}^{c}\backslash \{\veczer\}\big)=\\
& \proba_{\!\vecn}\big(\leproc(t_q)\in  \mathcal{D}^{c}\backslash \{\veczer\},T_{\sDelta}\le t_{q}\big)+
\proba_{\!\vecn}\big(\leproc(t_q)\in  \mathcal{D}^{c}\backslash \{\veczer\},T_{\sDelta}>t_{q}\big).
\end{align*}
Using the strong Markov property at time $T_{\sDelta}$ and \eqref{res2} we obtain
\begin{align*}
\MoveEqLeft \proba_{\!\vecn}\big(\leproc(t_q)\in  \mathcal{D}^{c}\backslash\{\veczer\},T_{\sDelta}\le t_{q}\big)\\
&= \E_{\vecn} \left[ \un_{\{T_{\ssDelta}\leq\, t_q\}} \proba_{\!\sleproc(T_{\ssDelta})}\left( \leproc(t_q-T_{\sDelta})\in  \mathcal{D}^{c}\backslash \{\veczer\}\right)\right]\\
& \leq C(1+t_q) \e^{-c K}.
\end{align*}
We bound the second term recursively in $q$.   
\begin{align*}
\MoveEqLeft[2] \proba_{\!\vecn}\big(\,T_{\sDelta}>t_{q},T_{\veczer}>T_{\sDelta}\big)\\
& = \E_{\vecn} \left[ \un_{\{T_{\ssDelta}>t_{q-1}\}} \un_{\{T_{\veczer}>T_{\ssDelta}\}}
\proba_{\!\sleproc(t_{q-1})}\left( T_{\sDelta}>t_1, T_{\veczer}>T_{\sDelta}\right)\right]\\
& \leq \delta \sup_{\vecn\, \in \Delta^{\!c}\backslash \{\veczer\}}\proba_{\!\vecn}\big(\,T_{\sDelta}>t_{q-1},T_{\veczer}>T_{\sDelta}\big)
\end{align*}
where we used the strong Markov property at time $t_{q-1}$ and \eqref{res1}. 
This implies
\[
\sup_{\vecn \in \Delta^{\!c}\backslash \{\veczer\}}\!\proba_{\!\vecn}\big(\leproc(t_q)\!\in\! \mathcal{D}^{c} 
\backslash \{\veczer\},T_{\sDelta}\!>t_{q}\big)
\le
\sup_{\vecn \in \Delta^{\!c}\backslash \{\veczer\}} \!\proba_{\!\vecn}\big(T_{\sDelta}>t_{q},T_{\veczer}\!>T_{\sDelta}\big) \le \delta^{q}.
\]
Therefore
\[
\sup_{\vecn\neq\veczer}\proba_{\!\vecn}\big(\leproc(t_q)\in  \mathcal{D}^{c}\backslash \{\veczer\}\big)\le
C\big(1+t_{q}\big)\e^{-c K}+\,\delta^{q}.
\]
Taking $q=\lfloor K \rfloor$ we conclude that there exists a constant $c'>0$ such that for $K$ large enough
\[
\sup_{\vecn\neq \veczer}\proba_{\!\vecn}\big(\leproc(t_{{\scriptscriptstyle \lfloor K \rfloor}})\in  \mathcal{D}^{c}
\backslash \{\veczer\}\big)\le\e^{-c' K}.
\]
This implies 
\[
\proba_{\!\snuk}\!\big(\leproc(t_{{\scriptscriptstyle \lfloor K \rfloor}})\in  \mathcal{D}^{c},T_{\veczer}
>t_{{\scriptscriptstyle \lfloor K \rfloor}}\big)\le \e^{-c' K}
\]
but by \eqref{pte-qsd}
\[
\proba_{\!\snuk}\!\big(\leproc(t_{{\scriptscriptstyle \lfloor K \rfloor}})\in  \mathcal{D}^{c},T_{\veczer}>t_{{\scriptscriptstyle \lfloor K \rfloor}}\big)=
\e^{-\lambda_{0}(K)\hspace{.4pt} t_{\scaleto{\lfloor K \rfloor}{4pt}}}\nuk\big(\mathcal{D}^{c}\big)
\]
and the result follows from \eqref{eq:lambda0}.
\end{proof}

\begin{corollary}\label{bornepuissance}
For each $q\in\integers$, there exists $C_{q}>0$ such that  for all $K$ large enough
\[
\int_{\mathcal{D}^c} \norm{\vecn}^{\scaleto{q}{4pt}}\dd \nuk(\vecn)\le C_{q}\, K^{q} \e^{-\pcte{borneexterieur} K}
\quad
\textup{and}
\quad
\int \norm{\vecn}^{\scaleto{q}{4pt}}\dd\nuk(\vecn)\le C_{q}\, K^{q}.
\]
\end{corollary}
\begin{proof}
It follows at once from \eqref{pte-qsd} (at time $1$)  and Theorem \ref{borneexp} that
\begin{equation}\label{momex}
\int \e^{\scaleto{\|}{7pt}\vecn\scaleto{\|}{7pt}_{\scaleto{1}{3pt}}}\!\dd\nuk(\vecn) \le \e^{\lambda_{0}(K)}\e^{\cte{borneexp}K}
\le 2\,\e^{\cte{borneexp}K}
\end{equation}
for $K$ large enough. We have
\begin{align*}
\int_{\mathcal{D}^c} \norm{\vecn}^{\scaleto{q}{4pt}}\dd\nuk(\vecn)
&= K^q \int_{\mathcal{D}^c} \left(\frac{\norm{\vecn}}{K}\right)^q \e^{-\frac{\scaleto{\|}{5pt}\vecn\scaleto{\|}{5pt}_{\scaleto{1}{3pt}}}{K}}
\e^{\frac{\scaleto{\|}{5pt}\vecn\scaleto{\|}{5pt}_{\scaleto{1}{3pt}}}{K}}\dd\nuk(\vecn)\\
& \leq  K^q q^q \e^{-q} \int \e^{\frac{\|\vecn\|_{\scaleto{1}{3pt}}}{K}} \un_{ \mathcal{D}^c}(\vecn)\dd\nuk(\vecn).
\end{align*}
We use H\"older inequality to get 
\[
\int_{ \mathcal{D}^c} \norm{\vecn}^{\scaleto{q}{4pt}}\dd\nuk(\vecn)
\leq 
K^q q^q \e^{-q} \left(\int \e^{\scaleto{\|}{7pt}\vecn\scaleto{\|}{7pt}_{\scaleto{1}{3pt}}} \!\dd\nuk(\vecn)\right)^{\frac{1}{K}} 
\left(\int \un_{ \mathcal{D}^c}(\vecn)\dd\nuk(\vecn)\right)^{1-\frac{1}{K}}.
\]
The first result follows from \eqref{momex} and Proposition \ref{borneexterieur}.
The second estimate follows from the first one, and the bound
$\sup_{\vecn \in  \mathcal{D}} \norm{\vecn} \leq \Oun K$.
\end{proof}

We now estimate centered moments. 
\begin{theorem}\label{bornemom}
For each $q\in\zentiers_{+}$, there exists $C_{q}>0$ such that for all $K$ large enough
\[
\int \|\vecn-K\vecxf \|_{\scaleto{2}{4pt}}^{\scaleto{2}{4pt}\scaleto{q}{5pt}}\dd\nuk(\vecn)\le C_{q} K^{q}.
\]
\end{theorem}

\begin{proof} The proof consists in a recursion over $q$. The bound is trivial for $q=0$. 
For $q\in\integers$ define the function
\[
f_q(\vecn)=\|\vecn-K\vecxf\|_{\scaleto{2}{4pt}}^{\scaleto{2q}{5.5pt}} \,\un_{\mathcal{D}_1}(\vecn)
\] 
where
\begin{equation*}
\mathcal{D}_1  =  \mathcal{B} \Big(\scaleto{K}{7.5pt}\vecxf, \scaleto{\frac{2K}{3}}{18pt}\min_{j}x^*_{j}\Big)\cap \domaine.
\end{equation*}
Recall that  $\vece{j}$  is the vector with $1$ at the $j$th coordinate and $0$ elsewhere. From the trivial identity
\begin{equation}
\label{trivial!}
\|\vecn - K\vecxf \pm \vece{j}\|_{\scaleto{2}{4pt}}^{\scaleto{2}{4pt}} = \|\vecn - K\vecxf\|_{\scaleto{2}{4pt}}^{\scaleto{2}{4pt}} \pm 2(n_j -K\vecxf_{j})+1
\end{equation}
it follows that
\begin{align*}
\MoveEqLeft[3] \big|\|\vecn - K\vecxf \pm \vece{j}\|_{\scaleto{2}{4pt}}^{\scaleto{2q}{5.3pt}} -
\|\vecn - K\vecxf\|_{\scaleto{2}{4pt}}^{\scaleto{2q}{5.3pt}} \pm 2q\, (n_j -K\vecxf_{j})\|\vecn - K\vecxf\|_{\scaleto{2}{4pt}}^{\scaleto{2q-2}{5.3pt}}\big|\\
& \qquad \qquad\leq 3^q 2^q \,(1+\|\vecn - K\vecxf\|_{\scaleto{2}{4pt}}^{\scaleto{2q-2}{5.3pt}}).
\end{align*}
Indeed, applying  the trinomial expansion to \eqref{trivial!}, we obtain
\begin{align*}
\MoveEqLeft
\big|\|\vecn - K\vecxf \pm \vece{j}\|_{\scaleto{2}{4pt}}^{\scaleto{2q}{5.3pt}} -
\|\vecn - K\vecxf\|_{\scaleto{2}{4pt}}^{\scaleto{2q}{5.3pt}} \pm 2q (n_j -K\vecxf_{j})\|\vecn - K\vecxf\|_{\scaleto{2}{4pt}}^{\scaleto{2q-2}{5.3pt}}\big|\\
& \leq 
q! \mathlarger{\sum}_{\substack{p_{1}\leq q-2\\ p_{1}+p_{2}+p_{3} = q}}
\frac{\|\vecn - K\vecxf\|_{\scaleto{2}{4pt}}^{\scaleto{2p_{1}}{5.3pt}} (2\hspace{1pt} \|\vecn - K\vecxf\|_{\scaleto{2}{4pt}})^{p_{2}} }{p_{1}! \, p_{2}! \, p_{3}!}
+ q\, \|\vecn - K\vecxf\|_{\scaleto{2}{4pt}}^{\scaleto{2q-2}{5.3pt}}.
\end{align*}
Observe that if $p_{1}\leq q-2, p_{1}+p_{2}+p_{3} = q$ and  then 
$2p_{1}+p_{2} =p_{1}+q-p_{3}\leq 2q -2 -p_{3}\leq 2q-2$, since $p_{3}\geq 0$.  This implies that
\[
\|\vecn - K\vecxf\|_{\scaleto{2}{4pt}}^{\scaleto{2p_{1}}{5.3pt}} (2 \|\vecn - K\vecxf\|_{\scaleto{2}{4pt}})^{\scaleto{p_{2}}{4pt}} 
\leq 2^q(1+\|\vecn - K\vecxf\|_{\scaleto{2}{4pt}}^{\scaleto{2q-2}{5.3pt}}).
\]
It follows that  
\begin{equation}\label{genfq}
\gen f_q(\vecn)=
2q K \sum_{j=1}^d X_j\left( \frac{\vecn}{K}\right) (n_j-Kx_j^*) \|\vecn - K\vecxf\|_{\scaleto{2}{4pt}}^{\scaleto{2q-2}{5.3pt}} \un_{\mathcal{D}_1}(\vecn)
+R_q(\vecn)
\end{equation}
where
\begin{equation}\label{Rqn}
|R_q(\vecn)|\leq \Oun\big( K 6^q (1+\|\vecn-K\vecxf\|_{\scaleto{2}{4pt}}^{\scaleto{2q-2}{5.3pt}}) \un_{\mathcal{D}_1}(\vecn) +  q K^{2q+1} \un_{\mathcal{D}^c}(\vecn)\big)
\end{equation}
To get this bound, we used the fact that
\[
\sup_{j=1,\ldots, d} |\un_{\mathcal{D}_1}(\vecn\pm \vece{j})-\un_{\mathcal{D}_1}(\vecn)|\leq \un_{\mathcal{D}^c}(\vecn).
\]
Using \eqref{Xalpha} we get
\begin{align}
\nonumber
\MoveEqLeft[6] K \sum_{j=1}^d X_j\left( \frac{\vecn}{K}\right) (n_j-Kx_j^{\scaleto{*}{3pt}}) \|\vecn - K\vecxf\|_{\scaleto{2}{4pt}}^{\scaleto{2q-2}{5.3pt}} \,\un_{\mathcal{D}_1}(\vecn)\\
& \leq -\beta'  \|\vecn - K\vecxf\|_{\scaleto{2}{4pt}}^{\scaleto{2q}{5.3pt}} \,\un_{\mathcal{D}_1}(\vecn) = -\beta'  f_q(\vecn)
\label{trucmusch}
\end{align}
where
\[
\beta'= \frac{\beta}{3} \min_j x_j^*\,.
\]
Integrating the equation \eqref{genfq} with respect to $\nuk$ and using \eqref{eq:miracle}, \eqref{Rqn}, \eqref{trucmusch} and Proposition \ref{borneexterieur},
we obtain
\[
(2q\beta'-\lambda_0(K))\,\nuk(f_q)\leq \Oun\big( K 6^q(1+\nuk(f_{q-1}))+ 6^q  K^{2q+1} \e^{-\pcte{borneexterieur} K}\big).
\]
Observing that $\nuk(f_0)\leq 1$, it follows by recursion over $q$ that, for each integer $q$, there exists $C'_q>0$
such that, for all $K$ large enough, $\nuk(f_q)\leq C'_q K^q$. Finally we have
\begin{align*}
\int \|\vecn-K\vecxf\|_{\scaleto{2}{4pt}}^{\scaleto{2q}{5.3pt}} \dd\nuk(\vecn)
& = \nuk(f_q) +\int \|\vecn-K\vecxf\|_{\scaleto{2}{4pt}}^{\scaleto{2q}{5.3pt}}\,\un_{\mathcal{D}_1^c}(\vecn) \dd\nuk(\vecn)\\
& \leq  \nuk(f_q) +\int\|\vecn-K\vecxf\|_{\scaleto{2}{4pt}}^{\scaleto{2q-2}{5.3pt}}\,\un_{\mathcal{D}^c}(\vecn)\dd\nuk(\vecn)
\end{align*}
since $\mathcal{D}\subset \mathcal{D}_1$.
The result follows using the previous estimate and Corollary \ref{bornepuissance}.
\end{proof}

The next result gives a more precise estimate for the average of $\vecn$ (instead of an error of order $\sqrt{K}\,$). 
\begin{proposition}\label{moyenne}
We have 
\[
\vmu^{\sK}-K\vecxf =\Oun
\]
where $\vmu^{\sK}$ is defined in \eqref{def-mup}.
Moreover,  since $\|\vecnf-K\vecxf \|_{\scaleto{2}{4pt}}=\Oun$, we have
\begin{equation}\label{nfvsxf}
\vmu^{\sK}-\vecnf=\Oun\,.
\end{equation} 
\end{proposition}
\begin{proof}
Define the functions 
\[
g_{j}(\vecn)=\langle\, \vecn-K \vecxf,\vece{j}\rangle,\; 1\le j\le d.
\]
By Taylor expansion and the polynomial bounds on $\vecB$ and $\vecD$ we get
\begin{align*}
& \gen g_{j}(\vecn)=K\big(B_{j}(\vecn/K)- D_{j}(\vecn/K)\big)\\
& =\sum_{m=1}^d \big(\partial_{m} B_{j}(\vecxf)-\partial_{m}
D_{j}(\vecxf)\big) g_{m}(\vecn)\un_{\mathcal{D}}(\vecn)+\Oun
\frac{\|\vecn-K\vecxf \|_{\scaleto{2}{4pt}}^{\scaleto{2}{4pt}}}{K}\, \un_{\mathcal{D}}(\vecn)\\
& \quad +\, \Oun\left(K^{p}+\normdeux{\vecn}^{\scaleto{p}{4pt}}\,\right)\un_{\mathcal{D}^{c}}(\vecn)
\end{align*}
for some positive integer $p$ independent of $K$.
Using Cauchy-Schwarz inequality, identity \eqref{eq:miracle}, Corollary \ref{bornepuissance} and Proposition \ref{borneexterieur} 
we get
\[
\int \big(1+\normdeux{\vecn}^{\scaleto{p}{4pt}}\big)\un_{\mathcal{D}^{c}}(\vecn)\dd\nuk(\vecn)=o(1).
\]
From Proposition \ref{borneexterieur}, Theorem \ref{bornemom} and \eqref{eq:lambda0} we  get
\[
\sum_{m=1}^d \big(\partial_{m} B_{j}(\vecxf)-\partial_{m}
D_{j}(\vecxf)\big)\nuk(g_{m})=\Oun.
\]
The result follows from the invertibility of the $d\times d$ matrix
$(\partial_{m} B_{j}(\vecxf)-\partial_{m}D_{j}(\vecxf))$ which follows
from \ref{H3}. The other inequalities follow immediately.
\end{proof}

\begin{corollary}\label{variance}  
For all $K>0$, we have
\[
\|\Sigma^{\sK}\|\leq \int \left\|\,\vecn-\vmu^{\sK}\right\|_{\scaleto{2}{4pt}}^{\scaleto{2}{4pt}} \dd \nuk(\vecn)=
\int \|\,\vecn-K\vecxf \|_{\scaleto{2}{4pt}}^{\scaleto{2}{4pt}}\dd\nuk(\vecn)+\Oun
\leq \Oun K.
\]
\end{corollary}
\begin{proof}
Combine Proposition \ref{moyenne} and Theorem \ref{bornemom}. 
\end{proof}

We now show that $\Sigma^{\sK}$ is indeed of order $K$. 
\begin{proposition}\label{borninf}
There exist two strictly positive constants $\pcte{borninf}$ and $\pcte{borninf}'$ such that for all $K$ large enough, the matrix $\Sigma^{\sK}$ satisfies
\[
\Sigma^{\sK}\ge \pcte{borninf} \, K \, \mathrm{Id}
\]
for the order among positive definite matrices, $\mathrm{Id}$ being the identity matrix, and, in particular,
\[
\int \left\|\,\vecn-\vmu^{\sK}\right\|_{\scaleto{2}{4pt}}^{\scaleto{2}{4pt}}\dd\nuk(\vecn)
\geq \pcte{borninf}' K.
\]
\end{proposition}
\begin{proof}
We denote by $\widetilde\Sigma^{\sK}$ the positive definite matrix
\[
\widetilde\Sigma^{\sK}_{p,q}= \int \big(n_{p}-n_{p}^{\scaleto{*}{3pt}}\big)
\big(n_{q}-n_{q}^{\scaleto{*}{3pt}}\big)\dd\nuk(\vecn)\,.
\]
By \eqref{nfvsxf} we have 
\begin{equation}\label{comparaisondescov}
\big\|\widetilde\Sigma^{\sK}-\Sigma^{\sK}\big\|_{\scaleto{2}{4pt}}=\Oun.
\end{equation}
Let $\vecv$ be a unit vector in $\real^{d}$.  We have
\[
\langle\, \vecv,\widetilde\Sigma^{\sK}\vecv\rangle=
\int \langle\, \vecv,(\vecn-\vecnf)\rangle^{\scaleto{2}{4pt}}\dd\nuk(\vecn)\ge
\int_{\sDelta}\langle\, \vecv,(\vecn-\vecnf)\rangle^{\scaleto{2}{4pt}}\dd\nuk(\vecn).
\]
From Lemma 5.3 in \cite{ccm2} there exists a constant $c>0$ such that for all $K$ large enough and all $\vecn\in \Delta$,
\[
{\nuk}(\{\vecn\})\geq c\, U_{\Delta}(\{\vecn\})
\]
where $U_{\Delta}$ is the uniform distribution on $\Delta$. 
Therefore
\[
\langle\, \vecv, \widetilde\Sigma^{\sK}\,\vecv\rangle\ge
c \int_{\Delta} \langle\, \vecv,(\vecn-\vecnf)\rangle^{\scaleto{2}{4pt}}
\dd U_{\Delta}(\vecn)
\]
and we get
\[
\langle\, \vecv,\tilde\Sigma^{\sK}\vecv\rangle\geq \pcte{borninf} K\|\vecv\|_{\scaleto{2}{4pt}}^{\scaleto{2}{4pt}}.
\]
The result follows.
\end{proof}

\section{Controlling time averages of the estimators}
\label{sec:process}

For $T>0$, we define the time average of a function $f:\domaine\to\real$ by 
\begin{equation}
\label{statf}
S_{f}(T,K)=\frac{1}{T}\int_{0}^{T}f(\leproc(s)) \dd s.
\end{equation}
The goal of this section is to obtain a control of $|S_{f}(T,K)-\nuk(f)|$ for a suitable class of functions. 

We recall the following result from \cite[Theorem 3.1]{ccm2}.
\begin{theorem}[\cite{ccm2}]\label{decor-thm}
There exist $a>0$, $K_0>1$ such that, for all 
$t\geq 0$ and for all $K\geq K_0$, we have
\begin{equation}\label{decor}
\sup_{\vecn\in\sdomaine\backslash\{\veczer\}}
\big\| \proba\!_{\vecn}(\leproc(t)\in \cdot \,,\, t< T_{\veczer}) -\proba\!_{\vecn}(t< T_{\veczer})
\,\nuk(\cdot)\big\|_{{\scriptscriptstyle \mathrm{TV}}} \leq 2\e^{-\frac{at}{\log K}}.
\end{equation}
\end{theorem}

It is also proved in \cite{ccm2} that, for a time much larger than $\log K$ and much smaller than the extinction time 
(which is of order $\exp(\Thetaun K)$), the law of the  process at time $t$  is close to the qsd. The accuracy of the 
approximation depends on the initial condition. This suggests to study the distance between the law of the process at time $t$ and 
the qsd as a function of the initial condition, $K$ and $t$. 
This will result  from \eqref{decor} if $\proba_{\!\vecn}\big(T_{\veczer}\leq t\big)$ can be estimated. In fact we prove a more general 
result. 

\begin{lemma}\label{sortpas}
For $\gamma\geq 0$, define $\tau_{\gamma}=\inf\big\{t\geq 0:\| \leproc(t)\|_{\scaleto{1}{4pt}}\le \gamma K\big\}.$
There exist $\delta>0$, $\alpha>0$ and $C>0$ such that for all $\vecn\in\domaine$, $K\ge 1$, 
$0\le \gamma\le 1\wedge \frac{\alpha}{\snorm{\vecxf}}$ and $t\ge0$, we have
\begin{align}
\nonumber
& \proba_{\!\vecn}\big(\tau_{\gamma}\leq t\big)
\leq C\Bigg( \exp\left(-\delta\left(\zeta\left(\frac{\norm{\vecn}}{K}\wedge\,\alpha\right)-\gamma \norm{\vecxf}\right)K\right)\\
& \hspace{90pt} +\, t \exp\left(-\delta\,(\alpha-\gamma  \norm{\vecxf}) K\right)\Bigg)
\label{probasortgen}
\end{align}
where
\begin{equation}\label{def-zeta}
\zeta=\min_{1\le j\le d}\;x^{\scaleto{*}{3pt}}_{j}>0.
\end{equation}
\end{lemma}
Taking $\gamma=0$ in \eqref{probasortgen}, we get 
\begin{equation}\label{probasort}
 \proba_{\!\vecn}\big(T_{\veczer}\leq t\big)
\leq C\left(\exp\left(-\delta\left(\zeta\frac{\snorm{\vecn}}{K}\wedge\,\alpha\right)K\right)+\, t \exp(-\alpha\,\delta K)\right).
\end{equation}

\begin{proof}
It follows from \ref{H1} and  \ref{H3} (using Taylor's expansion of $\vecX(\vecx)$ near $\veczer$) that there exists $\alpha_{0}\in(0,R)$ (where $R$ was introduced in Assumption \ref{H3})  such that for all $\vecx\in \rpd$ satisfying $\normdeux{\vecx}\le \alpha_{\scaleto{0}{4pt}}$ we have
\begin{align}
\nonumber
\langle \vecX(\vecx), \vecxf\rangle
& \geq \beta \scaleto{\|}{10pt}\vecxf\scaleto{\|}{10pt}^{\scaleto{2}{3.6pt}} \, \normdeux{\vecx} - 2\beta\,  \normdeux{\vecx} \langle \vecx,\vecxf\rangle + \beta  \normdeux{\vecx}^{\scaleto{3}{3.6pt}} + \langle \vecX(\vecx), \vecx \rangle \\
& \geq \beta\scaleto{\|}{10pt}\vecxf\scaleto{\|}{10pt}^{\scaleto{2}{3.6pt}}\,  \normdeux{\vecx} + \Oun \normdeux{\vecx}^{\scaleto{2}{3.6pt}}
\geq \frac{\beta\,\normdeux{\vecxf}^{\scaleto{2}{3.6pt}}}{2}\,\normdeux{\vecx}.
\label{soleilcouchant}
\end{align}
For $\alpha\in(0,\alpha_{0}]$ and $\delta>0$ to be chosen later on, we define
\[
\psi(\vecn)=\e^{-\delta(\langle \vecn,\vecxf\rangle\wedge\, \alpha K)}.
\]
It is easy to verify that if $\langle \vecn, \vecxf\rangle>\alpha\,K+\normdeux{\vecxf}$ we have
\[
\gen \psi(\vecn)=0.
\]
If $\alpha K-\normdeux{\vecxf}\le \langle \vecn, \vecxf\rangle\le \alpha K+\normdeux{\vecxf}$ we have
\[
\big|\gen \psi(\vecn)\big|\le {\cal O}(K) \e^{-\alpha\delta K}.
\]
For $\langle \vecn, \vecxf\rangle\le \alpha K-\normdeux{\vecxf}$, we have $\norm{\vecn}\le \langle \vecn,\vecxf\rangle/\zeta\le \alpha K/\zeta$, where $\zeta$ is defined in \eqref{def-zeta}, and 
\[
\gen \psi(\vecn) = K \mathfrak{g}\!\left(\delta, \frac{\vecn}{K}\right) \e^{-\delta \langle \vecn,\vecxf\rangle}
\]
where the function $\mathfrak{g}$ is defined by
\[
\mathfrak{g}(s,\,\vecx)=\sum_{j=1}^{d}B_{j}(\vecx)\big(\e^{-s x_{j}^*}-1\big)+\sum_{j=1}^{d}D_{j}(\vecx)\big(\e^{s x_{j}^*}-1\big).
\]
We have
\begin{align*}
\MoveEqLeft
\mathfrak{g}(s,\vecx)=-s\sum_{j=1}^{d}\big(B_{j}(\vecx)-D_{j}(\vecx)\big)x_{j}^*\\
& \;\;\,\qquad +\sum_{j=1}^{d}B_{j}(\vecx)\big(\e^{-s x_{j}^*}-1+s x_{j}^*\big)
+\sum_{j=1}^{d}D_{j}(\vecx)\big(\e^{s x_{j}^*}-1-s x_{j}^*\big).
\end{align*}
From the differentiability of the vector fields $\vecB$ and $\vecD$ and using \eqref{soleilcouchant}, it
follows that there exists a constant $\Gamma>0$ such that, for all $0\le s\le 1$ and  $\normdeux{\vecx}<\alpha_{0}$ we have
\begin{align*}
\mathfrak{g}(s,\vecx)
& = -s\, \langle\, \vecX(\vecx)\,,\, \vecxf\rangle+ \Oun\, s^{\scaleto{2}{4pt}}\, \normdeux{\vecx} \\
& \le -s\,\frac{\beta\,\normdeux{\vecxf}^{\scaleto{2}{3.6pt}}}{2}\,\normdeux{\vecx} +\Gamma s^{2}\, \normdeux{\vecx}.
\end{align*}
Therefore we can choose $\delta>0$ and $0<\alpha<\alpha_0$ such that
\[
\sup_{\snormdeux{\vecx}\le\, \alpha}\mathfrak{g}(\delta,\vecx)<0.
\]
Therefore, for all $\vecn$, we have
\[
\gen \psi(\vecn)\le \Oun K \e^{-\alpha\delta K}.
\]
For $\tilde{\gamma}>0$ (independent of $K$), we define
\[
\tilde{\tau}_{\tilde \gamma}=\inf\big\{t\geq  0 : \langle\,  \leproc(t),\vecxf\rangle\le \tilde \gamma K\big\}.
\]
We apply Ito's formula to $\psi$ to get
\[
\E_{\vecn}\big[\psi\big(\leproc(t\wedge\tilde{\tau}_{\tilde \gamma})\big)\big]
=\psi(\vecn)+\E_{\vecn}\left[\,\int_{0}^{t\wedge\tilde{\tau}_{\tilde \gamma}}\gen\psi(\leproc(s))\dd s\right].
\]
We have
\[
\tilde{\gamma}K-\zeta \leq \langle\, \leproc(\tilde{\tau}_{\tilde \gamma}),\vecxf\rangle \leq \tilde{\gamma} K
\]
hence
\[
\psi( \leproc(\tilde{\tau}_{\tilde \gamma}))\geq \e^{-\delta(\tilde{\gamma}\wedge\, \alpha)K} \e^{-\delta \zeta}.
\]
Then
\[
\E_{\vecn}\big[\psi\big(\leproc(t\wedge\tilde{\tau}_{\tilde \gamma})\big)\big]
\geq \proba_{\!\vecn}\big(\tilde{\tau}_{\tilde \gamma}\leq t)\e^{-\delta\,(\tilde \gamma\wedge\, \alpha) K}\e^{-\delta \zeta}.
\]
Therefore
\[
\proba_{\!\vecn}\big(\tilde{\tau}_{\tilde \gamma}\leq t)\e^{-\delta\,(\tilde \gamma\wedge\,\alpha) K}\e^{-\delta \zeta}
\le \e^{-\delta(\langle \vecn,\vecxf\rangle \wedge\,\alpha K)}+\,t\, \Oun K \e^{-\alpha\delta K}.
\]
To conclude, observe that 
\[
\proba_{\!\vecn}\big(\tau_{\gamma}\leq t)
\leq
\proba_{\!\vecn}\big(\tilde{\tau}_{\tilde \gamma}\leq t)
\]
for $\tilde{\gamma}=\gamma\, \norm{\vecxf}$ because for all $\vecn\in\domaine$,
\[
0<\zeta\, \norm{\vecn} \leq \langle \vecn,\vecxf \rangle
\leq \norm{\vecn}\sup_{j=1,\ldots,d} x^*_{j} \leq \norm{\vecn} \scaleto{\|}{10pt}\vecxf\scaleto{\|}{10pt}_{\scaleto{1}{4pt}}
\] 
and $\| \leproc(\tau_{\gamma})\|_{\scaleto{1}{4pt}}\leq \gamma K$.
\end{proof}
 
We have the following result.

\begin{proposition}
\label{pre-theorem-ergodic}
For all bounded functions $h:\domaine\to\real$, $t\ge 0$, $\vecn\in \domaine$, and $K>K_{0}$, we have
\[ 
\left|\E_{\vecn}\left[h\big(\leproc(t)\big)\right]-\nuk(h)\right|
\le \Oun \|h\|_{\scaleto{\infty}{3pt}}\Big(\e^{-\delta\left(\zeta\frac{\|\vecn\|_{\scaleto{1}{3pt}}}{K}\wedge\,\alpha\right) K}
+\, t \e^{-\alpha\delta K}+\e^{-\frac{a t}{\log K}}\Big)
\]
where $\alpha, \delta$ and $\zeta$ are defined in Lemma \ref{sortpas}, and $a$ and $K_0$ are defined in Theorem \ref{decor-thm}.
\end{proposition}

\begin{proof}
From the bound \eqref{decor} we get
\[
\left|\E_{\vecn}\left[h\big(\leproc(t)\big)\un_{\{T_{\veczer}>t\}}\right]
-\proba\!_{\vecn}(t<T_{\veczer})\,\nuk(h)\right|
\le\Oun\, \|h\|_{\scaleto{\infty}{3pt}}\e^{-\frac{a t}{\log K}}.
\]
This implies 
\begin{align*}
\MoveEqLeft \left|\E_{\vecn}\left[h\big(\leproc(t)\big)\right]-\nuk(h)\right| \\
& \le
\left|\E_{\vecn}\left[h\big(\leproc(t)\big)\un_{\{T_{\veczer}\le t\}}\right]\right|
+\proba\!_{\vecn}(t\ge T_{\veczer})\,\nuk(h)+\Oun\,  \|h\|_{\scaleto{\infty}{3pt}}\e^{-\frac{a t}{\log K}}
\\
& \le \Oun\, \|h\|_{\scaleto{\infty}{3pt}}\left(\proba\!_{\vecn}(t\ge T_{\veczer})+\e^{-\frac{a t}{\log K}}\right)
\\
& \le \Oun\,  \|h\|_{\scaleto{\infty}{3pt}}\left(\e^{-\delta\left(\zeta\frac{\|\vecn\|_{\scaleto{1}{3pt}}}{K}\wedge\,\alpha\right)K}
+t \e^{-\alpha\delta K}+\e^{-\frac{a t}{\log K}}\right)
\end{align*}
using \eqref{probasort}. 
\end{proof}

We now extend Proposition \ref{pre-theorem-ergodic} to more general functions.
For  $q\in\zentiers_+$, we define the Banach space
 $\mathscr{F}\!\!_{\sK\!,q}$ by
\begin{equation}\label{observables}
\mathscr{F}\!\!_{\sK\!,q}
=\left\{f: \domaine \rightarrow \mathds{R} :
\|f\|_{\sK\!,q}:= \sup_{ \vecn\neq \veczer} \frac{|f(\vecn)|}{K^q+\normdeux{\vecn}^{\scaleto{q}{3.6pt}}}<+\infty\right\}.
\end{equation}

We have the following result for time-averages of functions in $\mathscr{F}\!\!_{\sK}$.
\begin{theorem}\label{pseudoergod} 
For all $K>K_0$, $f\in\mathscr{F}\!\!_{\sK\!, q}$,  $T>0$, and $\vecn\in \domaine$, we have
\begin{align*}
\MoveEqLeft \left|\E_{\vecn}\big[S_{f}(T,K)\big]-\nuk(f)\right|
\leq \Oun\, \|f\|_{\sK\!,q}\,(K^{q}+\normdeux{\vecn}^{\scaleto{q}{3.6pt}})\\
& \times \left(\frac{1}{T} + \e^{-\delta\left(\zeta\frac{\|\vecn\|_{\scaleto{1}{3pt}}}{K}\wedge\,\alpha\right)K}
+T\e^{-\alpha\delta K}+\frac{\log K}{a T}+ \big(1-\e^{-\lambda_0(K)}\big)^{\frac{1}{2}} \right)
\end{align*}
where $\alpha, \delta$ and $\zeta$ are defined in Lemma \ref{sortpas}, and $\lambda_0(K)$ is defined in \eqref{eq:lambda0}. 
\end{theorem}
\begin{remark}
One can check that if one modifies slightly the definition of the time average \eqref{statf} by integrating from $1$ to $T+1$,
then one can remove the term $\normdeux{\vecn}^{\scaleto{q}{4pt}}$ from the previous estimate.
\end{remark}
\begin{proof}
For $f\in\mathscr{F}\!\!_{\sK\!, q}$,  Corollary \ref{bornepuissance} gives
\[
\big|\nuk(f)\big| \leq \Oun K^q \|f\|_{\sK\!,q}.
\]
By Proposition \ref{tempscourts} we have
\[
\left|\frac{1}{T}\,\E_{\vecn}
\left[\,\int_{0}^{1\wedge T}f\big(\leproc(s)\big)\dd s\right]\right|\le \Oun\, \|f\|_{\sK\!,q}\, \frac{K^{q}+\normdeux{\vecn}^{\scaleto{q}{3.6pt}}}{T}.
\]
Hence for $T\le 1$ we get
\[
\left|\E_{\vecn}\big[S_{f}(T,K)\big]-\nuk(f)\right|\le \Oun\,  \|f\|_{\sK\!,q}\big(K^{q}
+\normdeux{\vecn}^{\scaleto{q}{3.6pt}}\big)
\left(\frac{1}{T}+1\right).
\]
For $T>1$, we have by the Markov property that
\begin{align*}
\frac{1}{T}\,\E_{\vecn}\!\left[\int_{1}^{T}
f\big(\leproc(s)\big)\dd s\right]
&=\frac{1}{T}\,\int_{1}^{T}\E_{\vecn}\!\left[\E_{\sleproc(s-1)}\!\left[f\big(\leproc(1)\big)\right]\right]\dd s\\
& =\frac{1}{T}\,\int_{0}^{T-1}\E_{\vecn}\!\left[g\big(\leproc(s)\big)\right]\dd s
\end{align*}
where we set
\begin{equation}\label{def-g}
g(\vecm)=\E_{\vecm}\!\left[f\big(\leproc(1)\big)\right].
\end{equation}
By Corollary \ref{auxbornees}, the function $g$ is bounded and
\begin{equation}\label{ginfini}
\|g\|_{\infty}\le \Oun\, \|f\|_{\sK\!,q}\,K^{q}.
\end{equation}
Applying Proposition \ref{pre-theorem-ergodic} to  $g$ thus gives
\begin{align*}
\MoveEqLeft
\left|\E_{\vecn}\!\left[g\big(\leproc(s)\big)\right]-\nuk(g)\right|\\
& \le \Oun\, \|f\|_{\sK\!,q}\, K^q \left(\e^{-\delta\left(\zeta\frac{\|\vecn\|_{\scaleto{1}{3pt}}}{K}\wedge\,\alpha\right)K}
+s\e^{-\alpha\delta K}+\e^{-\frac{a s}{\log K}}\right). 
\end{align*}
Integrating over $s\in [0,T-1]$ yields
\begin{align*}
\MoveEqLeft
\left| \frac{1}{T}\,\int_{0}^{T-1}\E_{\vecn}\!\left[g\big(\leproc(s)\big)\right]\dd s - \frac{T-1}{T}\, \nuk(g)\right|\\
& \leq \Oun\, \|f\|_{\sK\!,q} \,\frac{K^q}{T}\left((T-1)\e^{-\delta\left(\zeta\frac{\|\vecn\|_{\scaleto{1}{3pt}}}{K}\wedge\,\alpha\right)K}
+\,(T-1)^2\e^{-\alpha\delta K}+\frac{\log K}{a}\right).
\end{align*}
Using Lemma \ref{f-et-g} (stated and proved right after this proof), we finally obtain
\begin{align*}
\MoveEqLeft
\left|\E_{\vecn}\big[S_{f}(T,K)\big]-\nuk(f)\right|\\
& \le  \Oun\, \|f\|_{\sK\!,q}\,\frac{K^{q}+\normdeux{\vecn}^{\scaleto{q}{3.6pt}}}{T}\\
& \quad +\Oun\, \|f\|_{\sK\!,q}  \frac{K^q}{T}\!
\left((T-1)\e^{-\delta\left(\zeta\frac{\|\vecn\|_{\scaleto{1}{3pt}}}{K}\wedge\,\alpha\right)K}+\, (T-1)^2\e^{-\alpha\delta K}+\frac{\log K}{a}\right)\\
& \quad +  \Oun\, \|f\|_{\sK\!,q}\, K^q \big(1-\e^{-\lambda_0(K)}\big)^{\frac{1}{2}} +  \frac{1}{T} \nuk(g) + 
\nuk(f)\un_{\{T\leq 1\}}\\
& \leq \Oun\, \|f\|_{\sK\!,q}\, \big(K^{q}+\normdeux{\vecn}^{\scaleto{q}{3.6pt}}\big)\Big(\frac{1}{T}\Big(2+\frac{\log K}{a }\Big) +
\e^{-\delta\left(\zeta\frac{\|\vecn\|_{\scaleto{1}{3pt}}}{K}\wedge\,\alpha\right)K} +\, T\!\e^{-\delta\,\alpha K}\\
& \quad + \big(1-\e^{-\lambda_0(K)}\big)^{\frac{1}{2}} + \un_{\{T\leq 1\}}\Big).
\end{align*}
This finishes the proof of the theorem.
\end{proof}

We used the following lemma in the previous proof.
\begin{lemma}
\label{f-et-g}
For $f\in \mathscr{F}\!\!_{\sK\!,q}$ and $g$ defined in \eqref{def-g} we have
\[
|\nuk(g)-\nuk(f)|\leq \Oun K^q\,  \|f\|_{\sK\!,q}\big(1-\e^{-\lambda_0(K)}\big)^{\frac{1}{2}}.
\]
\end{lemma}
\begin{proof}
We write
\[
\nuk(g)= \E_{\snuk}\!\big[f(\leproc(1)) \un_{\{T_{\veczer}>1\}}\big]
+
\E_{\snuk}\!\big[f(\leproc(1)) \un_{\{T_{\veczer}\leq 1\}}\big].
\]
Since $\nuk$ is a qsd, it follows by Cauchy-Schwarz inequality that 
\begin{align*}
\MoveEqLeft |\nuk(g)-\nuk(f)|\\
& \leq \big(1-\e^{-\lambda_0(K)}\big)\big|\nuk(f)\big|+ \!\left(\E_{\snuk}\!\big[f^2(\leproc(1))\big]\right)^{\frac{1}{2}} \!
\left(\E_{\snuk}\!\big[ \un_{\{T_{\veczer}\leq 1\}}\big] \right)^{\frac{1}{2}}\\
& \leq\Oun K^q\, \|f\|_{\sK\!,q} \big(1-\e^{-\lambda_0(K)}\big)^{\frac{1}{2}}
\end{align*}
where we used Corollaries \ref{auxbornees}  and \ref{bornepuissance} and the fact that under $\nuk$ the law of $T_{\veczer}
$ is exponential with  parameter $\lambda_0(K)$. The lemma is proved.
\end{proof}

\section{Fluctuation and correlation relations}
\label{chaleur}
\subsection{Proof of Theorem \ref{thm-onsager-relation}}\label{proof:onsager}

Let
\[
\widetilde{\Sigma}_{i,j}^{\sK}(t)=\E_{\snuk}\!\big[(N_{i}^{\sK}(t)-n_{i}^{\scaleto{*}{2.5pt}})(N_{j}^{\sK}(0)-n_{j}^*)\big],\; i,j=1,\ldots,d.
\]
For $1\leq i\leq d$, let $f_i(\vecn)=\langle \vecn-\vecnf,\vece{i}\rangle$.
We have, since $B_{i}(\vecxf)=D_{i}(\vecxf),\; 1\le i\le d$, and $\vecnf/K-\vecxf=\Oun/K$, 
\begin{align*}
& \frac{\dd}{\dd t}\widetilde{\Sigma}_{i,j}^{\sK}(t)\\
& =\E_{\snuk}\!\!\left[ \gen f_i(\leproc(t)) (N_{j}^{\sK}(0)-n_{j}^{\scaleto{*}{2.5pt}})\right]\\
& =K\E_{\snuk}\!\!\left[ B_{i}\!\left(\frac{\leproc(t)}{K}\right)\!\!\big(N_{j}^{\sK}(0)-n_{j}^{\scaleto{*}{2.5pt}}\big)\! \right]\!-K\E_{\snuk}\!\!\left[ D_{i}\!\left(\frac{\leproc(t)}{K}\right)\!\!\big(N_{j}^{\sK}
(0)-n_{j}^{\scaleto{*}{2.5pt}}\big)\!\right]\\
&=K\E_{\snuk}\!\!\left[ \left(B_{i}\left(\frac{\leproc(t)}{K}\right)-B_{i}\left(\frac{\vecnf}{K}
\right)\right)
\big(N_{j}^{\sK}(0)-n_{j}^{\scaleto{*}{2.5pt}}\big)\right]\\
& \quad -K\E_{\snuk}\!\!\left[ \left(D_{i}\left(\frac{\leproc(t)}{K}\right)
-D_{i}\left(\frac{\vecnf}{K}\right)\!\right)(N_{j}^{\sK}(0)-n_{j}^{\scaleto{*}{2.5pt}})\right] + \Oun.
\end{align*}
As in the previous proof, we split the integrals according to whether $\leproc(t)\in \mathcal{D}$ or
$\leproc(t)\in \mathcal{D}^c$. Using Cauchy-Schwarz inequality, Corollary \ref{bornepuissance}, and the fact that $\nuk$ is a qsd, the second contribution is exponentially small in $K$. In the first contribution, we use Taylor expansion around $\vecxf$.
The error terms are bounded by
\[
\frac{\Oun}{K}\, \E_{\snuk}\!\big[ \|\leproc(t)-K\vecxf\|_{\scaleto{2}{4pt}}^{\scaleto{2}{4pt}}\,\|\leproc(0)-K\vecxf\|_{\scaleto{2}{4pt}}\big]
+\Oun.
\]
Now we use Cauchy-Schwarz inequality, Theorem \ref{bornemom} and that $\nuk$ is a qsd to obtain
\begin{align*}
& \frac{\dd}{\dd t}\widetilde{\Sigma}_{i,j}^{\sK}(t)\\
& =\sum_{\ell=1}^{d}\left(\partial_{\ell}B_{i}\!\left(\vecxf\right)-\partial_{\ell}D_{i}\left(\vecxf\right)\right)
\E_{\snuk}\!\big[ (N_{\ell}^{\sK}(t)-n_{\ell}^{\scaleto{*}{2.5pt}}) (N_{j}^{\sK}(0)-n_{j}^{\scaleto{*}{2.5pt}})\big]\!+\!
\mathcal{O}(\sqrt K)\\
& =\sum_{\ell=1}^{d}M_{i,\ell}^{*}\, \widetilde{\Sigma}^{\sK}_{\ell,j}(t)+\mathcal{O}\big(\sqrt K\,\big).
\end{align*}
Since $M^*$ has a spectrum contained in the open left half-plane by \ref{H3}, we integrate the equation 
\[
\frac{\dd}{\dd t}\widetilde{\Sigma}^{\sK}(t)= M^*\widetilde{\Sigma}^{\sK}(t)+\mathcal{O}\big(\sqrt K\,\big)
\]
from $0$ to $\tau$ using the method of constant variation and obtain
\begin{equation*}
\widetilde{\Sigma}^{\sK}(\tau)=\e^{\tau M^*}\widetilde{\Sigma}^{\sK}(0)+\mathcal{O}\big(\sqrt K\,\big).
\end{equation*}
We arrive at the desired relation by using \eqref{nfvsxf}.

\subsection{Proof of Theorem \ref{thm-Kubo}}\label{proof:kuboeq}

Recall that 
\[
\Sigma^{\sK}_{p,q}=\Sigma^{\sK}_{p,q}(0)=\int \big(n_{p}-\mu_{p}^{\sK}\big) \big(n_{q}-\mu_{q}^{\sK}\big) \dd\nu\!_{\sK}(\vecn).
\]
We will first do the proof with the following matrix instead of $\Sigma^{\sK}$:
\[
\widetilde\Sigma^{\sK}_{i,j}=\int (n_{i}-n_{i}^{\scaleto{*}{2.5pt}})(n_{j}-n_{j}^{\scaleto{*}{2.5pt}}) \dd \nuk(\vecn).
\]
On the one hand we have by \eqref{eq:miracle} 
\[
\big\langle \gen^{\dagger}\nuk,(n_{i}-n^{\scaleto{*}{2.5pt}}_{i})(n_{j}-n^{\scaleto{*}{2.5pt}}_{j})\big\rangle=-\lambda_{0}(K)
\big\langle\nuk,(n_{i}-n^{\scaleto{*}{2.5pt}}_{i})(n_{j}-n^{\scaleto{*}{2.5pt}}_{j})\big\rangle.
\]
By Theorem \ref{bornemom} and \eqref{eq:lambda0} the right-hand side of this equation is exponentially small in $K$. 
On the other hand,  using formula \eqref{gene} we have
\begin{align*}
\MoveEqLeft
\big\langle \gen^{\dagger}\nuk,(n_{i}-n^{\scaleto{*}{2.5pt}}_{i})(n_{j}-n^{\scaleto{*}{2.5pt}}_{j})\big\rangle=
\big\langle \nuk,\gen\big((n_{i}-n^{\scaleto{*}{2.5pt}}_{i})(n_{j}-n^{\scaleto{*}{2.5pt}}_{j})\big)\big\rangle\\
& =K\sum_{\ell=1}^{d}\Big\langle\nuk,B_{\ell}\left(\frac{\vecn}{K}\right)\big((n_{j}-n^{\scaleto{*}{2.5pt}}_{j})\,\delta_{i,\ell}+
(n_{i}-n^{\scaleto{*}{2.5pt}}_{i})\,\delta_{j,\ell} +\delta_{i,\ell}\,\delta_{j,\ell}\big)\Big\rangle\\
& \quad +K\sum_{\ell=1}^{d}\Big\langle\nuk,D_{\ell}\left(\frac{\vecn}{K}\right)\big(-(n_{j}-n^{\scaleto{*}{2.5pt}}_{j})\delta_{i,\ell}-(n_{i}-n^{\scaleto{*}{2.5pt}}_{i})
\delta_{j,\ell}+\delta_{i,\ell}\delta_{j,\ell}\big)\Big\rangle\\
& =K\Big\langle\nuk,\Big(B_{i}\left(\frac{\vecn}{K}\right)-D_{i}\left(\frac{\vecn}{K}\right)\Big)(n_{j}-n^{\scaleto{*}{2.5pt}}_{j})\Big\rangle\\
& \quad +K\Big\langle\nuk,\Big(B_{j}\left(\frac{\vecn}{K}\right)-D_{j}\left(\frac{\vecn}{K}\right)\Big)\,(n_{i}-n^{\scaleto{*}{2.5pt}}_{i})
\Big\rangle\\
& \quad +K \Big\langle\nuk,B_{i}\left(\frac{\vecn}{K}\right)+D_{i}\left(\frac{\vecn}{K}\right)\Big\rangle\, \delta_{i,j}.
\end{align*}
We split each integral by separating integration over $\mathcal{D}$ (defined in \eqref{boule}) and $\mathcal{D}^c$.
Inside $\mathcal{D}^c$, we apply Corollary \ref{bornepuissance} and use the assumption that $B$ and $D$ are polynomially 
bounded. Inside $\mathcal{D}$, we use Taylor's formula around $\vecxf$ for the functions $B_i(\vecn/K)-D_i(\vecn/K)$, and
$B_i(\vecn/K)+D_i(\vecn/K)$. We also use that $B_{i}(\vecxf)=D_{i}(\vecxf),\; 1\le i\le d$, and $\vecnf/K-\vecxf=
\mathcal{O}\big(\frac{1}{K}\big)$. The error terms are then bounded by
\[
\mathcal{O}\left(\frac{1}{K}\right)\int \|\vecn-K\vecxf\|_{\scaleto{2}{4pt}}^{\scaleto{3}{4pt}} \dd \nuk(\vecn)\quad\text{and}\quad
\Oun  \int \|\vecn-K\vecxf\|_{\scaleto{2}{4pt}} \dd \nuk(\vecn)
\]
respectively. Using Theorem \ref{bornemom}, both bounds are of order $\sqrt{K}$.
We obtain
\[
\sum_{\ell=1}^{d}M_{i,\ell}^{*}\,\widetilde{\Sigma}^{\sK}_{\ell,j}+
\sum_{\ell=1}^{d}M_{j,\ell}^*\,\widetilde{\Sigma}^{\sK}_{\ell,i}+2K B_{i}\left(\vecxf\right)\delta_{i,j}
=\mathcal{O}\big(\sqrt K\,\big)
\]
which can be written in the more compact form
\begin{equation}\label{fluceq}
M^* \widetilde{\Sigma}^{\sK}+\widetilde{\Sigma}^{\sK} {M^*}^{\intercal}+2\EuScript{D}^{\sK}=\mathcal{O}\big(\sqrt K\,\big)
\end{equation}
where $\EuScript{D}^{\sK}$ is the diagonal matrix of averages birth (or death) rates. To finish the proof, it remains to
replace $\widetilde\Sigma^{\sK}$ by $\Sigma^{\sK}$. This is done by using \eqref{comparaisondescov}.
\begin{remark}
Note that each term on the left hand side is of  order $K$, see Corollary \ref{variance}.
\end{remark}
\begin{remark} We will see in Appendix C that the qsd $\,\nu_{K}\,$ around $\vecnf$ is well approximated at scale $\sqrt{K}$ by a Gaussian distribution. 
Dividing out \eqref{fluceq} by $2K$ and taking the limit $K\to\infty$, we recover  Relation \eqref{pierre}, as expected from Theorem \ref{confinement}. 
\end{remark}

\section{Variance estimates for the estimators}
\label{sec:stat}

It is straightforward to apply Theorem \ref{pseudoergod} to $S^{{\scriptscriptstyle\vmu}}(T,K)$,
$S^{{\scriptscriptstyle C}}(T,\tau,K)$, $S^{{\scriptscriptstyle \EuScript{D}}}(T,K)$, and 
$S^{{\scriptscriptstyle \Sigma}}(T,K)$, which are defined respectively in \eqref{stat-mu},  \eqref{stat-sigma}, \eqref{stat-D}, 
and \eqref{stat-MSigma}. This gives the bound \eqref{control-est-mu} on $S^{{\scriptscriptstyle\vmu}}(T,K)$ anounced in  
Section \ref{sec:intro}. The bounds for the other estimators all have the same structure. We will not state them.

In this section we prove two variance estimates for any time average $S_{f}(T,K)$ with $f\in \mathscr{F}\!\!_{\sK\!,q}$. 
In the first one, one starts from anywhere in $\domaine$, while in the second one the starting distribution is the qsd.
Recall that $S^{{\scriptscriptstyle \Sigma}}(T,K)=S^{{\scriptscriptstyle C}}(T,0,K)$. We will only give the proofs of these
estimates for $S^{{\scriptscriptstyle \Sigma}}(T,K)$, since manipulating $S^{{\scriptscriptstyle C}}(T,\tau,K)$ is cumbersome but 
otherwise the proofs are the same.

\begin{proposition}\label{varcorn}
There exist strictly positive constants $\delta', \zeta', \alpha', \theta', C'$ and $K_0\geq 2$ such that,  for all $K\geq K_0$,
$f\in\mathscr{F}\!\!_{\sK\!,q}$ (see Definition \ref{observables}), $T\geq 0$, and $\vecn\neq \veczer$,  we have
\begin{align*}
\MoveEqLeft
\E_{\vecn}\Big[\big(S_{f}(T,K)-\nuk(f)\big)^{\scaleto{2}{4pt}}\Big]
\leq C' \|f\|_{\sK\!,q}^{\scaleto{2}{4pt}}(c_q\norm{\vecn}^{\scaleto{q}{3.8pt}} +K^q)\\
& \times
\left(
\frac{\norm{\vecn}^{\scaleto{q}{3.6pt}} +K^q \log K}{T\vee 1}+ 
K^q  \e^{-\delta'\big(\zeta' \frac{\|\vecn\|_{\scaleto{1}{3pt}}}{K}\wedge \,\alpha'\big)K}+ T K^q \e^{-\theta' K} 
\right)
\end{align*}
where $c_q$ was defined in Proposition \ref{tempscourts}.
\end{proposition}

One can use Chebyshev inequality to bound $\proba_{\!\vecn}\big(\big|S_{f}(T,K)-\nuk(f)\big|>\delta\big)$,
for any $\delta>0$.

The proof of Proposition \ref{varcorn} is postponed to Appendix \ref{proof-variances}.
The previous estimate, as well as all the estimates we will give below, have the same behaviour in their dependence in $K$, $\vecn$ and
$T$. They display the qualitative behaviour that we met several times:
\begin{enumerate}
\item The bounds are not useful for $K$ too small.
\item If $K$ is large, the bounds are  not useful if $\vecn$ is small (order one) because the process can be absorbed at
$\veczer$ in a time of order one with a sizeable probability.
\item Finally, for $K$ large and $\vecn$ of order $K$,  the time $T$ must be large enough (polynomial in $K$ in
our bounds) but not too large (less than an exponential in $K$ because the process can reach the origin with high probability in 
such large times).  
\end{enumerate}
Integrating the previous estimate with respect to the qsd, we get the following control.

\begin{corollary}\label{varcornu}
There exist two positive constants $C''>0$ and $\theta''$ such that  
for all $K\geq K_0$, for all $f\in\mathscr{F}\!\!_{\sK\!,q}$ and for all $T\geq 0$, we have
\begin{align*}
\MoveEqLeft \E_{\snuk}\!\Big[\big(S_{f}(T,K)-\E_{_{\snuk}}\big(f\big)\big)^{\scaleto{2}{4pt}}\Big]\leq C''  \|f\|_{\sK\!,q}^{\scaleto{2}{4pt}}\, K^{\scaleto{2}{4pt}q}\\
& \times \left((1+C_{2q})(1+c_q)\frac{\log K}{T\vee 1}+ (1+C_q)(1+T) \e^{-\theta'' \! K} \right)
\end{align*}
where $K_0$ is as in the previous proposition, $c_q$ is defined in Proposition \ref{tempscourts}, and $C_q$ is defined in
Corollary \ref{bornepuissance}.
\end{corollary}
Observe that the previous inequality is only useful in the range $0\leq T \le \e^{\theta'' \! K}$. 
The proofs of the two previous estimates are postponed to Appendix \ref{proof-variances}.

We now apply the previous results to our estimators.

\begin{proposition}
For all $1\leq p\leq d$, we have
\begin{align*}
\MoveEqLeft[6]
\E_{\vecn}\Big[\big(S^{{\scriptscriptstyle\vmu}}_p(T,K)-\mu_p^{\sK}\big)^{\scaleto{2}{4pt}}\Big]\le  \Oun(c_1\norm{\vecn} +K)\\
& \times
\left(\frac{\norm{\vecn} +K \log K}{T\vee 1}+ K\e^{-\delta'\big(\zeta' \frac{\|\vecn\|_{\scaleto{1}{3pt}}}{K}\wedge\, \beta'\big)K} + T K\e^{-\theta K}\right)
\end{align*}
and
\[
\E_{\snuk}\Big[\big|S^{{\scriptscriptstyle\vmu}}_p(T,K)-\mu_p^{\sK}\big|^{\scaleto{2}{4pt}}\Big]\le
\Oun K^{\scaleto{2}{4pt}}\!\left(\frac{1+\log K}{T\vee 1}+(1+T) \e^{-\theta''\! K} \right).
\]
\end{proposition}

\begin{proof}
The proof follows by applying Proposition \ref{varcorn} and Corollary \ref{varcornu} to the functions $f(\vecn)=n_j$, $1\leq j\leq d$, which belong to $\mathscr{F}\!\!_{\sK\!,\scaleto{1}{4pt}}$.
\end{proof}

\begin{proposition}
\label{vs2}
For $1\leq p,p'\leq d$ and for all $\vecn\neq \veczer$, we have 
\begin{align*}
\MoveEqLeft[6]
\E_{\vecn}\Big[\big(S^{{\scriptscriptstyle \Sigma}}_{p,p'}(T,K)-\Sigma_{p,p'}^{\sK}\big)^{\scaleto{2}{4pt}}\Big]\le  
\Oun \big(c_{\scaleto{2}{4pt}}\normdeux{\vecn}^{\scaleto{2}{3.5pt}} +K^{\scaleto{2}{4pt}}\big)^{\scaleto{2}{4pt}}\\
& \times
\left(
\frac{1+\log K}{T\vee 1}+ \e^{-\delta'\big(\zeta' \frac{\|\vecn\|_{\scaleto{1}{3pt}}}{K}\wedge\, \beta'\big)K} + T  \e^{-\theta K} 
\right)
\end{align*}
and
\[
\E_{\snuk}\!\left[\big(S^{{\scriptscriptstyle \Sigma}}_{p,p'}(T,K)-\Sigma_{p,p'}^{\sK}\big)^{\scaleto{2}{4pt}}\right]
\le \Oun K^{\scaleto{4}{4pt}} \left(\frac{1+\log K}{T\vee 1}+(1+T) \e^{-\theta'' \!K} \right).
\]
\end{proposition}

\begin{proof}
The proof follows by applying Proposition \ref{varcorn} and Corollary \ref{varcornu} to the functions
$f(\vecn)=n_p n_{p'}$, $1\leq p,p' \leq d$,  which belong to $\mathscr{F}\!\!_{\sK\!,\scaleto{2}{4pt}}$.
\end{proof}

\begin{proposition}
There exist positive constants $\widetilde{C}, \tilde{\theta}$, $\tilde{\delta}$, $\tilde{\zeta}$ and $\tilde{\beta}$ such that for all $K\geq 2$, $T>0$ and $1\leq \ell\leq d$,
\begin{align*}
\MoveEqLeft \E_{\vecn}\big[\big(S^{{\scriptscriptstyle \EuScript{D}}}_{\ell}(T,K)-K\,B_{\ell}(\vecxf)\big)^{\scaleto{2}{4pt}}\Big]\le\\
& \widetilde{C} \Big(
K + \frac{A_\ell (1+C_{q_\ell}) K}{T} + K^{1-q_\ell} \frac{A_\ell}{T} (K+\norm{\vecn})^{q_\ell} \mathfrak{R}_\ell \\
&  \qquad \qquad +  K^{{\scaleto{2}{4pt}}-{\scaleto{2}{4pt}}q_\ell} A_\ell^{\scaleto{2}{4pt}}
(K+\norm{\vecn})^{\scaleto{2}{4pt}q_\ell}(\mathfrak{R}_\ell^{\scaleto{2}{4pt}}+\mathfrak{R}_\ell)
\Big)
\end{align*}
where 
\[
\mathfrak{R}_\ell=
(1+c_{q_\ell}) \left(\frac{1+\log K}{T} + T \e^{-\tilde{\theta} K}
+ \e^{-\tilde{\delta}\left(\tilde{\zeta}\frac{\|\vecn\|_{\scaleto{1}{3pt}}}{K}\wedge\, \tilde{\beta}\right)K}\right)
\]
and $A_\ell>0$, $q_\ell\in\integers$, are such that, for all $\vecx \in \rpd$, 
\[
|B_\ell(\vecx)|\leq A_\ell (1+ \norm{\vecx}^{\scaleto{q_\ell}{3.6pt}}).
\]
The existence of $A_\ell$ and $q_\ell$ follows from the assumptions on $B$. The constants $C_{q_\ell}$ and $c_{q_\ell}$ are 
defined in Corollary \ref{bornepuissance} and Lemma \ref{tempscourts}, respectively.\newline
We also have
\[
\E_{\snuk}\left[\big(S^{{\scriptscriptstyle \EuScript{D}}}_{\ell}(T,K)-K\,B_{\ell}(\vecxf)\big)^{\scaleto{2}{4pt}}\right]\le 
\]
\[
\widetilde{C} \left(K + \frac{A_\ell (1+C_{q_\ell}) K}{T} + \frac{A_\ell}{T} K \widetilde{\mathfrak{R}}_\ell +
K^{\scaleto{2}{4pt}} A_\ell^{\scaleto{2}{4pt}} ( \widetilde{\mathfrak{R}}_\ell^{\scaleto{2}{4pt}}+\widetilde{\mathfrak{R}}_\ell)\right)
\]
where
\[
\widetilde{\mathfrak{R}}_\ell=(1+C_{q_\ell}) \left((1+C_{\scaleto{2}{4pt}q_\ell})(1+c_{q_\ell})\frac{\log K}{T} +(1+T) \e^{-\tilde{\theta} K} \right)\,.
\]
\end{proposition}

\begin{proof}
First observe that
\[
S^{{\scriptscriptstyle \EuScript{D}}}_{\ell}(T,K)=\frac{\mathcal{N}^{\sK}_\ell(0,T)}{T}
\]
where $\mathcal{N}^{\sK}_\ell(0,T)$ is defined in Appendix \ref{nbsauts}. 
By assumption, the function $f_\ell(\vecn)=K^{q_\ell}\, B_\ell\left( \frac{\vecn}{K}\right) \in \mathscr{F}\!\!_{\sK\!,q_\ell}$.
Let $\mathfrak{m}$ be any probability measure on $\domaine$ having all its moments finite. We apply Theorem \ref{pseudoergod} to the function $f_\ell$, and then using integration against $\mathfrak{m}$ we get
\[
\left| \E_{\mathfrak{m}}\big[S_{f_\ell}(T,K)\big] - \nuk(f_\ell)\right|
\leq \Oun \|f_\ell\|_{\sK\!,q_\ell}
\]
\[
\times \int \left((K+\normdeux{\vecn})^{q_\ell} 
\left(\e^{-\delta\left(\zeta\frac{\|\vecn\|_{\scaleto{1}{3pt}}}{K}\wedge\,\beta\right) K}
+T\e^{-\delta\,\beta\,K}+\frac{1+\log K}{T}\right)\right) \dd \mathfrak{m}(\vecn).
\]
We now apply the identity in Proposition \ref{prop-rentree} and divide by $K^{q_\ell -1}$. We obtain
\begin{align}
\label{boborne}
\MoveEqLeft \left| \E_{\mathfrak{m}}\big[S^{{\scriptscriptstyle\EuScript{D}}}_{\ell}(T,K)\big] - \nuk\left(K B_\ell\Big(\frac{\vecn}{K}\Big)\right) \right|
\leq \Oun\|f_\ell\|_{\sK\!,q_\ell}\, K^{1-q_\ell}\\
& \times \mathlarger{\int} \left((K+\normdeux{\vecn})^{q_\ell} \left(\e^{-\delta\left(\zeta\frac{\|\vecn\|_{\scaleto{1}{3pt}}}{K}\wedge\,\beta\right)K}+T\e^{-\delta \beta K}+\frac{1+\log K}{T}\right)\right) \dd\mathfrak{m}(\vecn).
\nonumber
\end{align}
We now estimate 
\[
\int B_\ell\Big(\frac{\vecn}{K}\Big)\dd \nuk(\vecn)= \int_{\mathcal{D}} B_\ell\Big(\frac{\vecn}{K}\Big) \dd \nuk(\vecn)+
\int_{\mathcal{D}^c} B_\ell\Big(\frac{\vecn}{K}\Big) \dd \nuk(\vecn).
\]
The second integral is bounded from above by $\Oun/K$ using the polynomial bound on $B_\ell$ and the first estimate in Corollary \ref{bornepuissance}.
For the first integral we use Taylor expansion around $\vecxf$ to first order, then Cauchy-Schwarz inequality, and finally Theorem \ref{bornemom} for $q=1$. Therefore we obtain
\begin{align}
\nonumber
& \left| 
\E_{\mathfrak{m}}\big[S^{{\scriptscriptstyle \EuScript{D}}}_{\ell}(T,K)\big] - K B_\ell(\vecxf)\right|
\leq
\Oun \sqrt{K} +\Oun\, \|f_\ell\|_{\sK\!,q_\ell} K^{1-q_\ell}\\
& \times \int (K+\normdeux{\vecn})^{q_\ell} \left(\e^{-\delta\left(\zeta\frac{\|\vecn\|_{\scaleto{1}{3pt}}}{K}\wedge\,\beta\right)K}
+T\e^{-\delta\beta K}+\frac{1+\log K}{T}\right) \dd\mathfrak{m}(\vecn).
\label{pasbeau}
\end{align}
Now we apply the estimate in Proposition \ref{prop-rentree} to obtain
\begin{align*}
& \E_{\mathfrak{m}} \Big[\big(S^{{\scriptscriptstyle\EuScript{D}}}_{\ell}(T,K) -\E_{\mathfrak{m}}\big[S^{\EuScript{D}}_{\ell}(T,K)\big]\big)^{\scaleto{2}{4pt}}\Big]\\
& =\frac{1}{T^2} \, \E_{\mathfrak{m}} \Big[\big(\mathcal{N}^{\sK}_\ell(0,T) -
\E_{\mathfrak{m}}\big[\mathcal{N}^{\sK}_\ell(0,T)\big]\big)^{\scaleto{2}{4pt}}\Big]\\
& \leq \E_{\mathfrak{m}} \left[\left( 
\frac{1}{T} \int_0^T K^{1-q_\ell}\,  f_\ell(\leproc(s)) \dd s -\E_{\mathfrak{m}}\!\left[\frac{\mathcal{N}^{\sK}_\ell(0,T)}{T}\right]
\right)^{\scaleto{2}{4pt}}\,\right]\\
& \qquad + \frac{2}{T} \, \E_{\mathfrak{m}}\! \left[
\frac{1}{T} \int_0^T K^{1-q_\ell} \, f_\ell(\leproc(s)) \dd s 
\right] \\
& \leq 2\,\E_{\mathfrak{m}}\! \left[\left( 
\frac{1}{T} \int_0^T K^{1-q_\ell}\,  f_\ell(\leproc(s)) \dd s - K^{1-q_\ell} \nuk(f_\ell)
\right)^{\scaleto{2}{4pt}}\,\right] \\
& \qquad + 2\,\E_{\mathfrak{m}}\! \bigg[\left(K^{1-q_\ell} \nuk(f_\ell)-\E_{\mathfrak{m}}\!\left[S^{\EuScript{D}}_{\ell}(T,K)\right]
\right)^{\scaleto{2}{4pt}}\bigg]\\
& \qquad + \frac{2}{T} \, \E_{\mathfrak{m}}\! \left[\frac{1}{T} \int_0^T K^{1-q_\ell} \, f_\ell\big(\leproc(s)\big) \dd s 
\right].
\end{align*}
For the first term we use either Corollary \ref{varcornu} or Proposition \ref{varcorn}.  For the second term we use
\eqref{boborne}. For the third and last term we apply Theorem \ref{pseudoergod}, integrate with respect to $\mathfrak{m}$ and use
\eqref{pasbeau}. To finish the proof, we replace $\mathfrak{m}$ by either $\delta_{\vecn}$ or $\nuk$.
\end{proof}

Recall that $B_{p}(\vecxf)=D_{p}(\vecxf)$, $1\leq p\leq d$.
\begin{proposition}
\label{vs4}
Under the assumptions of Proposition \ref{varcorn} and Corollary \ref{varcornu}, we have, for all $1\leq p,p'\leq d$, and $\tau\geq 0$, 
\begin{align*}
\MoveEqLeft \E_{\vecn}\Big[\big(S^{{\scriptscriptstyle C}}_{p,p'}(T,\tau,K)-\Sigma^{\sK}_{p,p'}(\tau)\big)^{\scaleto{2}{4pt}}\Big]
\leq \Oun \big(c_2\norm{\vecn}^{\scaleto{2}{3.6pt}} +K^{\scaleto{2}{4pt}}\big)^{\scaleto{2}{4pt}}\times\\
&
\quad \left(\frac{1+\tau+\log K}{T\vee 1}+ \e^{-\delta'\big(\zeta' \frac{\|\vecn\|_{\scaleto{1}{3pt}}}{K}\wedge\, \beta'\big)K}
+(T+\tau)  \e^{-\theta K} 
\right).
\end{align*}
and
\begin{align*}
\MoveEqLeft
\E_{\snuk}\Big[\big(S^{{\scriptscriptstyle C}}_{p,p'}(T,\tau,K)-\Sigma^{\sK}_{p,p'}(\tau)\big)^{\scaleto{2}{4pt}}\Big] \\
& \leq \Oun K^{\scaleto{4}{4pt}} \!\left(\frac{1+\tau+\log K}{T}+ (1+T+\tau)\e^{-\theta''\! K} \right)
\end{align*}
\end{proposition}

\begin{proof}
The proof requires some simple modifications of the proofs of Propostions \ref{varcorn} and  \ref{varcornu}. This is left to the reader.
\end{proof}

\begin{remark}
If one modifies slightly the definition of the estimator by integrating from time $1$,
then, in the four previous propositions, one can replace the factor $(\norm{\vecn}+K)$ by $K$, and the factor
$(\norm{\vecn}^{\scaleto{2}{3.6pt}}+K^{\scaleto{2}{4pt}})$ by $K^{\scaleto{2}{4pt}}$.
\end{remark}

Recall that we defined in Section \ref{sec:intro} an empirical matrix $M^{*}_{\scriptscriptstyle{\mathrm{emp}}}(T,\tau,K)$ by
\[
\e^{\tau M^{*}_{\scriptscriptstyle{\mathrm{emp}}}(T,\tau,K)}=S^{{\scriptscriptstyle C}}(T,\tau,K)\,S^{{\scriptscriptstyle\Sigma}}(T,K)^{-1}
\]
and an empirical resilience by
\[
\rho^{\scaleto{*}{3pt}}_{\scriptscriptstyle{\mathrm{emp}}}(T,\tau,K)=- \sup\{ \operatorname{Re} (z): z\in\mathrm{Sp}\big(M^{*}_{\scriptscriptstyle{\mathrm{emp}}}(T,\tau,K)\big)\}.
\]
From the above results one can derive various statistical estimates for the difference between $\rho^{\scaleto{*}{3pt}}_{\scriptscriptstyle{\mathrm{emp}}}(T,\tau,K)$ 
and $\res$. We have the following result which was stated at the end of Section \ref{subsec:main-results}. As already mentioned, we 
use the symbol $\ll$ which is not rigorously defined to formulate a more transparent bound. 
The reader can easily step through the proof to get a more precise, but rather cumbersome bound.
Let us also note that the dependence on the initial state $\vecn$ is related to the part ``with a probability larger than 
$1-1/K$'' of the statement. Indeed, the estimate of this probability results from Chebychev inequality and variance estimates in which the process is 
started in $\vecn$.

\begin{theorem}\label{rhoemp}
For $\tau=\Theta(1)$, $\vecn=\Theta(K)$ (initial state) and $0<T\ll \exp(\Thetaun K)$,
and $K$ large enough, we have
\[
\big|\rho^{\scaleto{*}{3pt}}_{\scriptscriptstyle{\mathrm{emp}}}(T,\tau,K)-\res\big|\le
\Oun \left( \frac{K^{\scaleto{2}{4.5pt}}}{\sqrt{T}} +\frac{1}{\sqrt{K}}\right)
\]
with a probability higher than $1-1/K$.
In particular, if $T\geq K^{\scaleto{5}{4.5pt}} $, we have
\[
\big|\rho^{\scaleto{*}{3pt}}_{\scriptscriptstyle{\mathrm{emp}}}(T,\tau,K)-\res\big|\leq \Oun/\sqrt{K}.
\]
\end{theorem}

\begin{proof}
It follows from Propositions \ref{vs2} and \ref{vs4} and the standing assumptions that, with a probability higher that $1-1/K$, we have
\[
\|S^{{\scriptscriptstyle C}}(T,\tau,K)-\Sigma^{\sK}(\tau)\|\leq \Oun\, \frac{K^{\scaleto{3}{4pt}}}{\sqrt{T}}
\]
and
\[
\|S^{{\scriptscriptstyle \Sigma}}(T,K)-\Sigma^{\sK}\|\leq \Oun \, \frac{K^{\scaleto{3}{4.5pt}}}{\sqrt{T}}\,.
\]
($\|\cdot\|$ stands for any matrix norm on $\mathds{R}^{d\times d}$ since they are all equivalent.)
We now use Theorem \ref{thm-onsager-relation} and Proposition \ref{borninf} to obtain
\[
\Big\| \e^{\tau M^{*}_{\scriptscriptstyle{\mathrm{emp}}}(T,\tau,K)}-\e^{\tau M^{*}}\Big\|\leq \Oun \left( \frac{1}{\sqrt{K}}+ \frac{K^{\scaleto{2}{4pt}}}{\sqrt{T}}\right)\,.
\]
The result follows since $\tau$ is of order one.
\end{proof}

\appendix

\section{Proof of the two variance estimates}\label{proof-variances}

\subsection{Starting from anywhere: proof of Proposition \ref{varcorn}}

It is enough to prove the result for $\|f\|_{\sK\!,q}=1$.
We have
\[
\E_{\vecn}\!\left[\left(\frac{1}{T}
\int_{0}^{T}f\big(\leproc(t)\big)\dd t\right)^{\!\scaleto{2}{4pt}}\,\right] 
=\frac{\scaleto{2}{6pt}}{T^{\scaleto{2}{4pt}}}\int_{0}^{T}\!\!\!\dd t_{2}\int_{0}^{t_{2}}\!\E_{\vecn}\Big[
f\big(\leproc(t_{1})\big) f\big(\leproc(t_{2})\big)\!\Big]\!\dd t_{1}.
\]
\textbf{Step 1} is to estimate the contribution of the range $0\le t_{1}\le t_{2}\le 1$. Using Cauchy-Schwarz inequality and Proposition \ref{tempscourts} we get
\[
\left|\int_{0}^{1}\dd t_{2}\int_{0}^{t_{2}}\E_{\vecn}\Big[
f\big(\leproc(t_{1})\big) f\big(\leproc(t_{2})\big)\Big]\dd t_{1}\right|
\le \Oun \big(\norm{\vecn}^{\scaleto{q}{4pt}}+K^{q}\big)^{\scaleto{2}{4pt}}.
\]
\textbf{Step 2} is to estimate the contribution in the range $0\le
t_{2}-1\leq t_{1}\leq t_{2}$. This implies that $T>1$.
We have using again Proposition \ref{tempscourts}
\begin{align*}
\MoveEqLeft 
\left|\int_{1}^{T} \dd t_{2}\int_{t_{2}-1}^{t_{2}}\E_{\vecn}
\Big[f\big(\leproc(t_{1})\big) f\big(\leproc(t_{2})\big)\Big]\dd t_{1}\right|\\
& \le \int_{1}^{T} \dd t_{2}\int_{t_{2}-1}^{t_{2}}\left(\E_{\vecn}\Big[
f\big(\leproc(t_{1})\big)^{\scaleto{2}{4pt}}\Big]+ \E_{\vecn}\Big[f\big(\leproc(t_{2})\big)^{\scaleto{2}{4pt}}\Big]\right)\dd t_{1}\\
& \le \Oun \,T\,\Big(\norm{\vecn}^{\scaleto{q}{4pt}}+K^{q}\Big)^{\scaleto{2}{4pt}}.
\end{align*}
\textbf{Step 3}\\
(1) Using the Markov property and the definition of $g$ (see \eqref{def-g}) we have
\begin{align*}
\MoveEqLeft \int_{1}^{T} \dd t_{2}\int_{0}^{t_{2}-1}\E_{\vecn}\Big[
f\big(\leproc(t_{1})\big) f\big(\leproc(t_{2})\big)\Big]\dd t_{1}\\
& =\int_{0}^{T-1}  \dd s\int_{0}^{s}\E_{\vecn}\Big[
f\big(\leproc(t_{1})\big) g\big(\leproc(s)\big)\Big]\dd t_{1}\\
& = \int_{0}^{T-1}  \dd s\int_{0}^{s}
\E_{\vecn}\Big[f\big(\leproc(t_{1})\big) \E_{\sleproc(t_{1}) } \big[ g\big(\leproc(s-t_1)\big)\big]\Big]\dd t_{1}\,.
\end{align*}
Let us first write
\[
\E_{\vecn}\Big[f\big(\leproc(t_{1})\big) \E_{\sleproc(t_{1}) } \big[ g\big(\leproc(s-t_1)\big)\big]\Big]
\]
as the sum of $J_{1}(\vecn)$ and $J_{2}(\vecn)$ where
\[
J_1(\vecn)=\E_{\vecn}\Big[f\big(\leproc(t_{1})\big)\E_{\sleproc(t_{1}) } \big[ \un_{\{T_{\veczer} >s-t_1\}}g\big(\leproc(s-t_1)\big)\big]\Big]
\]
and
\[
J_2(\vecn)=\E_{\vecn}\Big[
f\big(\leproc(t_{1})\big) \E_{\sleproc(t_{1}) } \big[ 
\un_{\{T_{\veczer} \leq s-t_1\}}g\big(\leproc(s-t_1)\big)\big]\Big].
\]
We further decompose $J_1(\vecn)$ as $J_{1,1}(\vecn)+J_{1,2}(\vecn)$ where
\[
J_{1,1}(\vecn)=\E_{\vecn}\Big[
f\big(\leproc(t_{1})\big) \,\un_{\{T_{\veczer}\leq t_1\}}
 \E_{\sleproc(t_{1}) } 
\big[ \un_{\{T_{\veczer} >s-t_1\}}g\big(\leproc(s-t_1)\big)\big]\Big]
\]
and
\[
J_{1,2}(\vecn)=\E_{\vecn}\Big[
f\big(\leproc(t_{1})\big)\,\un_{\{T_{\veczer}> t_1\}}
\E_{\sleproc(t_{1}) } \big[ \un_{\{T_{\veczer} >s-t_1\}}g\big(\leproc(s-t_1)\big)\big]\Big].
\]
Since $\veczer$ is an absorbing state, we have for all $\vecn\neq \veczer$ that
\[
J_{1,1}(\vecn)=0.
\]
(2) We start by estimating $J_2(\vecn)$.
Since $\veczer$ is an absorbing state, we have
\[
J_2(\vecn)= g(\veczer)\, \E_{\vecn}\Big[
f\big(\leproc(t_{1})\big)\proba_{\!\sleproc(t_{1}) } \big( 
T_{\veczer} \leq s-t_1\big)\Big].
\]
Note that $g(\veczer)= \E_{\veczer}[f(\leproc(1))]= f(\veczer)$. 
Since we are going to use Lemma \ref{sortpas},  we write $J_2(\vecn)=J_{2,1}(\vecn)+J_{2,2}(\vecn)$
where
\[
J_{2,1}(\vecn)=
f(\veczer)\,\E_{\vecn}\Big[
f\big(\leproc(t_{1})\big) \un_{\{\|\leproc(t_1)\|_{\scaleto{1}{4pt}}> K\alpha/\zeta\}} \proba_{\!\leproc(t_{1}) } \big( 
T_{\veczer} \leq s-t_1\big)\Big].
\]
and
\[
J_{2,2}(\vecn)= f(\veczer)\, 
\E_{\vecn}\Big[
f\big(\leproc(t_{1})\big) \un_{\{\|\leproc(t_1)\|_{\scaleto{1}{4pt}}\leq K\alpha/\zeta\}} \proba_{\!\leproc(t_{1}) } \big( 
T_{\veczer} \leq s-t_1\big)\Big].
\]
We first estimate $J_{2,1}(\vecn)$. 
Using \eqref{ginfini}, Lemma \ref{sortpas} with $\gamma=0$, 
and since $f$ belongs to $\mathscr{F}\!\!_{\sK\!,q}$ (see \eqref{observables}), we have 
\begin{align*}
|J_{2,1}(\vecn)|
& \leq \Oun  \E_{\vecn}\big[|f(\leproc(t_1))|\big]  \e^{- \alpha \delta K}(1+C {(s-t_1)})\\
& \leq \Oun   (\norm{\vecn}^{\scaleto{q}{4pt}} + K^q)  \e^{- \alpha \delta K}(1+C {(s-t_1)})
\end{align*}
where we used Proposition \ref{tempscourts} for the second inequality. 

We now estimate $J_{2,2}(\vecn)$ by splitting it as
$J_{2,2,1}(\vecn)+J_{2,2,2}(\vecn)$ where
\[
J_{2,2,1}(\vecn)=f(\veczer)\,\E_{\vecn}\Big[f\big(\leproc(t_{1})\big) \un_{\left\{\|\leproc(t_1)\|_{\scaleto{1}{4pt}}\leq K\alpha/\zeta\right\}} \, 
\un_{\mathcal{E}_K}\,\proba_{\!\sleproc(t_{1}) } \big(T_{\veczer} \leq s-t_1\big)\Big].
\]
and
\[
J_{2,2,2}(\vecn)= f(\veczer)\,
\E_{\vecn}\!\left[
f\big(\leproc(t_{1})\big) \un_{\left\{\|\leproc(t_1)\|_{\scaleto{1}{4pt}}\leq K\alpha/\zeta\right\}}
\un_{\mathcal{E}_K^c}\proba_{\!\sleproc(t_{1}) } \big(T_{\veczer} \leq s-t_1\big)\right]
\]
where 
\[
\mathcal{E}_K:=\left\{\|\leproc(t_1)\|_{\scaleto{1}{4pt}} > \left(\left[\frac{1}{2\|\vecxf\!\|_{\scaleto{1}{4pt}}}
\left(\zeta\frac{\norm{\vecn}}{K}\wedge \alpha\right)\right]\wedge 1\right){K}\right\}.
\]
Proceeding as before we get
\begin{align*}
|J_{2,2,1}(\vecn)|
& \leq \Oun K^{q}\E_{\vecn}\left[
\un_{\mathcal{E}_K}\,\proba_{\!\sleproc(t_{1}) } \big( \un_{\{T_{\veczer} \leq s-t_1\}}\big)\right]\\
& \leq \Oun  K^{q} 
\Big(\e^{-\delta K\left( \Big(\big[\frac{1}{2\|\svecxf\!\|_{\scaleto{1}{3pt}}} \big(\zeta\frac{\|\vecn\!\|_{\scaleto{1}{3pt}}}{K}\wedge\, \alpha\big)\big]\wedge 1\Big)\wedge\, \alpha\right)}+(s-t_1) \e^{-\alpha\delta K}\Big).
\end{align*}
We used Lemma \ref{sortpas} with $\gamma=0$.\newline
We now handle $J_{2,2,2}(\vecn)$. 

Note that $\gamma\leq 1\wedge \frac{\alpha}{\|\vecxf\|_{\scaleto{1}{3pt}}}$.
We proceed as before with $f$ and $g$, and we use Lemma \ref{sortpas} with 
\[
\gamma=\left(\frac{1}{2\|\vecxf\|_{\scaleto{1}{4pt}}} \left(\zeta\frac{\|\vecn\|_{\scaleto{1}{4pt}}}{K}\wedge \alpha\right)\right)\wedge 1.
\]
to get
\begin{align*}
|J_{2,2,2}(\vecn)|
& \leq \Oun\, K^{q}\,\proba_{\!\vecn}\!\left(
\|\leproc(t_1)\|_{\scaleto{1}{4pt}} \le \left(\bigg(\frac{1}{2\|\vecxf\!\|_{\scaleto{1}{4pt}}}
\Big(\zeta\frac{\norm{\vecn}}{K}\wedge \alpha\Big)\bigg)\wedge 1\right)\,{K}\right)\\
& \leq \Oun K^q
\left( \e^{-{\delta} \big(\frac{1}{2} \big(\zeta\frac{\|\vecn\|_{\scaleto{1}{3pt}}}{K}\wedge\, \alpha\big)\wedge \|\vecxf\!\|_{\scaleto{1}{4pt}}\big)K}
+\, C\, t_1 \e^{-\frac{\alpha\delta K}{2}}\right).
\end{align*}
\noindent (3) 
Let us now estimate $\left| J_{1,2}(\vecn) -\nuk(f)^{\scaleto{2}{4pt}}\right|$ for all $\vecn\neq \veczer$.
We have 
\begin{align*}
& \left|J_{1,2}(\vecn) -\nuk(f)^{\scaleto{2}{4pt}}\right| \leq  \left|J_{1,2}(\vecn)
-\nuk(g)\, \E_{\vecn}\,\big[ f(\leproc(t_1)) \, \un_{\{T_{\veczer} >t_1\}} \, \proba_{\!\sleproc(t_{1})}\big( T_{\veczer} >s-t_1\big)\big]\right|\\
& + \left|\nuk(g) \E_{\vecn}\Big[f\big(\leproc(t_{1})\big)\un_{\{T_{\veczer}>t_{1}\}}
\proba_{\!\sleproc(t_{1})}\big(T_{\veczer}>s-t_{1}\big)\Big]-\nuk(g) \,\E_{\vecn}\Big[f\big(\leproc(t_{1})\big)\un_{\{T_{\veczer}>t_{1}\}}\Big]\right|\\
& +\left| \nuk(g)\,\E_{\vecn}\Big[f\big(\leproc(t_{1})\big)
\un_{\{T_{\veczer}>t_{1}\}}\Big]-\nuk(g)\nuk(f)\right|+ 
\left|\nuk(g)\nuk(f) - \nuk(f)^{\scaleto{2}{4pt}}\right|\\
&= W_{1}(\vecn)+W_{2}(\vecn)+W_{3}(\vecn)+W_{4}.
\end{align*}

(3)-(i) By Theorem \ref{decor-thm} and since $\leproc(t_1)\neq \veczer$, we have
\[
\Big| 
\E_{\sleproc(t_{1})} \big[ \un_{\{T_{\veczer} >s-t_1\}}g\big(\leproc(s-t_1)\big)\big]
- \proba_{\!\leproc(t_{1})}\big( T_{\veczer} >s-t_1\big)\, \nu(g)
\Big|
\leq \Oun\, K^q  \e^{-\frac{a(s-t_{\scaleto{1}{3pt}})}{\log K}}.
\]
Hence, using Proposition \ref{tempscourts}, we get for all $\vecn\neq \veczer$
\[
 W_{1}(\vecn)\leq
\Oun \, K^q \big(c_q \norm{n}^{\scaleto{q}{4pt}} + K^q\big) \e^{- \frac{a(s-t_{\scaleto{1}{3pt}})}{\log K}}.
\]

(3)-(ii) We have 
\[
\left|\nuk(g) \E_{\vecn}\Big[
f\big(\leproc(t_{1})\big)\un_{\{T_{\veczer}>t_{1}\}}
\proba_{\!\sleproc(t_{1})}\big(T_{\veczer}>s-t_{1}\big)\Big]-
\nuk(g) \E_{\vecn}\Big[f\big(\leproc(t_{1})\big)\un_{\{T_{\veczer}>t_{1}\}}\Big]\right|
\]
\[
\le \big|\nuk(g)\big| \E_{\vecn}\Big(
\big|f\big(\leproc(t_{1})\big)\big|\;\un_{\{T_{\veczer}>t_{1}\}}
\proba_{\!\sleproc(t_{1})}\big(T_{\veczer}\le s-t_{1}\big)\Big)
\]
Define $0<\gamma'<\beta$ by
\[
\gamma'=\gamma'(\vecn)=\frac{1}{2} \left(\frac{\zeta \norm{\vecn}}{K} \wedge \alpha\right).
\]
We split the right hand side in two terms:
\begin{align*}
\MoveEqLeft 
\E_{\vecn}\Big[
\big|f\big(\leproc(t_{1})\big)\big|\un_{\{T_{\veczer}>t_{1}\}}
\proba_{\!\sleproc(t_{1})}\big(T_{\veczer}\le s-t_{1}\big)\Big]\\
& =\E_{\vecn}\Big[\un_{\{\|\sleproc(t_{1})\|_{\scaleto{1}{4pt}}\le \gamma' K\}}
\big|f\big(\leproc(t_{1})\big)\big|\un_{\{T_{\veczer}>t_{1}\}}
\proba_{\!\sleproc(t_{1})}\big(T_{\veczer}\le s-t_{1}\big)\Big] \\
& \quad +\E_{\vecn}\Big[\un_{\{\|\sleproc(t_{1})\|_{\scaleto{1}{4pt}}> \gamma' K\}}
\big|f\big(\leproc(t_{1})\big)\big|\un_{\{T_{\veczer}>t_{1}\}}
\proba_{\!\sleproc(t_{1})}\big(T_{\veczer}\le s-t_{1}\big)\Big].
\end{align*}
The  first term is estimated using the growth property of $f$, Lemma \ref{sortpas}, 
and Cauchy-Schwarz inequality, namely 
\begin{align*}
\MoveEqLeft \E_{\vecn}\Big[\un_{\{\|\sleproc(t_{1})\|_{\scaleto{1}{4pt}}\le \gamma' K\}}\big|f\big(\leproc(t_{1})\big)\big|\,
\un_{\{T_{\veczer}>t_{1}\}}\,\proba_{\!\sleproc(t_{1})}\big(T_{\veczer}\le s-t_{1}\big)\Big]\\
& \le \E_{\vecn}\Big[\un_{\{\|\sleproc(t_{1})\|_{\scaleto{1}{4pt}}\le \gamma' K\}}
\big|f\big(\leproc(t_{1})\big)\big|^{\scaleto{2}{4pt}}\Big]^{\frac{1}{2}}\proba_{\!\vecn}\big(\|\leproc(t_{1})\|_{\scaleto{1}{4pt}}\le \gamma'K\big)^{\frac{1}{2}}\\
& \le \Oun K^{q}\,\proba_{\!\vecn}\big(\|\leproc(t_{1})\|_{\scaleto{1}{4pt}}\le\gamma'K\big)^{\frac{1}{2}} \\
& \le \Oun K^{q}\Big( \e^{-\frac{\delta}{2}\left(\zeta\frac{\|\vecn\|_{\scaleto{1}{3pt}}}{K}\wedge\,\alpha\right)K}
+\,\Oun\, t_1 \e^{-\frac{\alpha\delta K}{2}} \Big)^{\frac{1}{2}}. 
\end{align*}
To deal with the second term, we observe using Lemma \ref{sortpas} and Proposition \ref{tempscourts} that, 
if $\|\leproc(t_{1})\|_{\scaleto{1}{4pt}}> \gamma' K$, then
\begin{align*}
\proba_{\!\sleproc(t_{1})}\big(T_{\veczer}\le s-t_{1}\big)
& \le  \e^{-\delta\left(\zeta\frac{\|\sleproc(t_{\scaleto{1}{3pt}})\|_{\scaleto{1}{3pt}}}{K}\wedge\,\alpha\right)K}+\,\Oun (s-t_{1})\e^{-\alpha\delta K}\\
& \leq \e^{-\delta\left(\zeta\gamma'\wedge\,\alpha\right)K}+\,\Oun (s-t_{1})\e^{-\alpha\delta K}\\
& =\e^{-\delta\left(\zeta\left(\frac{1}{2} \left(\frac{\zeta \|\vecn\|_{\scaleto{1}{3pt}}}{K} \wedge\, \alpha \right)\right)\wedge\,\alpha\right)K}+\,\Oun (s-t_{1})\e^{-\alpha\delta K}.
\end{align*}
Now
\begin{align*}
\MoveEqLeft \E_{\vecn}\Big[\un_{\{\|\sleproc(t_{1})\|_{\scaleto{1}{4pt}}> \gamma' K\}}
\big|f\big(\leproc(t_{1})\big)\big|\un_{\{T_{\veczer}>t_{1}\}}
\proba_{\!\sleproc(t_{1})}\big(T_{\veczer}\le s-t_{1}\big)\Big]\\
& \leq \Oun \E_{\vecn}\Big[\un_{\{\|\sleproc(t_{1})\|_{\scaleto{1}{4pt}}> \gamma' K\}}
\big|f\big(\leproc(t_{1})\big)\big|\un_{\{T_{\veczer}>t_{1}\}}
\Big]\\
&  \quad \;\times \Big( \e^{-\delta\left(\zeta\left(\frac{1}{2} \left(\frac{\zeta \|\vecn\|_{\scaleto{1}{3pt}}}{K} \wedge\, \alpha\right)\right)
\wedge\,\alpha\right)K}+\,\Oun (s-t_{1})\e^{-\alpha\delta K}\Big)\\
& \leq \Oun \Big(\norm{\vecn}^{\scaleto{q}{4pt}}+K^{q}\Big) \Big( \e^{-\delta\left(\zeta\left(\frac{1}{2} \left(\frac{\zeta \|\vecn\|_{\scaleto{1}{3pt}}}{K} 
\wedge\, \alpha\right)\right)\wedge\,\alpha\right)K}+\,\Oun (s-t_{1})\e^{-\alpha\delta K}\Big)\\
&
\quad \; \times \Big( \e^{-\delta \frac{(1\wedge \zeta)}{2} \left(\frac{\zeta^2 \|\vecn\|_{\scaleto{1}{3pt}}}{2K} \wedge\, \alpha\right)K}+\,
\Oun (s-t_{1})\e^{-\alpha\delta K}\Big).
\end{align*}

(3)-(iii) Let us now prove that for all $\vecn\neq \veczer$,
\begin{align}
\nonumber
W_{3}(\vecn)
&= \left|\E_{\vecn}\Big(f\big(\leproc(t_{1})\big)\un_{\{T_{\veczer}>t_{1}\}}\Big)-\nuk(f)\right| \\
\nonumber
& \leq \Oun (c_q \norm{\vecn}^{\scaleto{q}{4pt}} + K^q)\Big(\e^{-\frac{a (t_{\scaleto{1}{3pt}}-1)}{\log K}} +  \e^{-\lambda_0(K)}
\e^{-\delta\big(\zeta\frac{\|\vecn\|_{\scaleto{1}{3pt}}}{K}\wedge\, \alpha \big)K} \\
\label{A1}
&  \hskip 4cm + \,C(t_1-1)\e^{-\alpha\delta K}+1-\e^{-\lambda_0(K)}\Big).
\end{align}
For $0\leq t_1\leq 1$, using Proposition \ref{tempscourts} we obtain 
\[
\left|
\E_{\vecn}\Big[f\big(\leproc(t_{1})\big)\un_{\{T_{\veczer}>t_{1}\}}\Big]\right|
\leq \Oun (c_q\norm{\vecn}^{\scaleto{q}{4pt}} + K^q).
\]
We now deal with $t_1>1$.
The Markov property gives
\begin{align*}
\E_{\vecn}\Big[f\big(\leproc(t_{1})\big)\un_{\{T_{\veczer}>t_{1}\}}\Big]
& = \E_{\vecn}\Big[\un_{\{T_{\veczer}>t_{1}-1\}}
\E_{\sleproc(t_{1}-1)}\big[f\big(\leproc(1)) \un_{\{T_{\veczer}>1\}}\big]\Big]\\
& = \E_{\vecn}\Big[\un_{\{T_{\veczer}>t_{1}-1\}}
\tilde g(\leproc(t_{1}-1))\Big]
\end{align*}
where
\[
\tilde g (\vecn) = \E_{\vecn}\big[\leproc(1) \un_{\{T_{\veczer}>1\}}\big] \leq g(\vecn)
\] 
is a function bounded by $\Oun  K^q$. 
For $\vecn\neq \veczer$, we use Theorem \ref{decor-thm} and Corollary \ref{auxbornees} to get
\begin{align*}
&\Big|
\E_{\vecn}\Big[\un_{\{T_{\veczer}>t_{1}-1\}}\, 
\E_{\sleproc(t_{1}-1)}\big[f\big(\leproc(1)) \un_{\{T_{\veczer}>1\}}\big]
\Big] \\
& \quad\quad-\proba_{\!\vecn}\big( T_{\veczer}>t_1-1\big) \E_{\snuk}\big[ f\big(\leproc(1)) \un_{\{T_{\veczer}>1\}}\big]
\Big| \\
& \leq\Oun\, K^q\, \e^{\frac{a(t_{\scaleto{1}{3pt}}-1)}{\log K}}.
\end{align*}
Since $\nuk$ is the qsd, we have
\[
 \E_{\snuk}\big[ f\big(\leproc(1)) \un_{\{T_{\veczer}>1\}}\big]= \e^{-\lambda_0(K)} \nuk(f).
\]
Using Corollary \ref{bornepuissance}, Lemma \ref{sortpas} and the properties of $f$ we obtain
\begin{align*}
\MoveEqLeft 
\left|\proba_{\!\vecn}\big( T_{\veczer}>t_1-1\big)\, \E_{\snuk}\big[ f\big(\leproc(1)) \un_{\{T_{\veczer}>1\}}\big]
-\nuk(f)\right|\\
&
\leq \Oun K^q\Big( \e^{-\lambda_0(K)}\e^{-\delta\big(\zeta\frac{\|\vecn\|_{\scaleto{1}{3pt}}}{K}\wedge\, \alpha \big)K} +\, C(t_1-1)\e^{-\beta\delta K}+
1-\e^{-\lambda_0(K)}\Big)
\end{align*}
and \eqref{A1} is proved. 

(3)-(iv) Let us note that
\[
W_{4} \leq |\nuk(f)| \,|\nuk(f)-\nuk(g)|.
\]
Proposition \ref{tempscourts} and  Lemma \ref{f-et-g} give
\[
W_{4} \leq  \Oun K^{\scaleto{2q}{6pt}}\big(1-\e^{-\lambda_0(K)}\big)^{\frac{1}{2}}.
\]

(3)-(v) Collecting the informations given in the four previous estimates, we obtain a precise estimation 
of $\left|J_{1,2}(\vecn) -\nuk(f)^{\scaleto{2}{4pt}}\right|$ for all $\vecn\neq \veczer$.

(3)-(vi)
We have 
\begin{align*}
\MoveEqLeft
\left|
\E_{\vecn}\Big[f\big(\leproc(t_{1})\big) \E_{\sleproc(t_{1}) } \big[ g\big(\leproc(s-t_1)\big)\big]\Big]
-\nuk(f)^{\scaleto{2}{4pt}}
\right|\\
& \leq |J_{2}(\vecn)| + \left|J_{1,2}(\vecn)  -\nuk(f)^{\scaleto{2}{4pt}}
\right|.
\end{align*}
Collecting the above relevant estimates we obtain that there exist $\delta',\zeta', \beta', \theta'$ (all being positive and independent
of $K$) such that
\begin{align*}
\MoveEqLeft
\left|
\E_{\vecn}\Big[f\big(\leproc(t_{1})\big) \E_{\sleproc(t_{1}) } \big[ g\big(\leproc(s-t_1)\big)\big]\Big]-\nuk(f)^{\scaleto{2}{4pt}}
\right| \leq \Oun K^q (c_q\norm{\vecn}^{\scaleto{q}{4pt}} +K^q)\\
&
\times \left(
\un_{\{t_1\leq 1\}}+ \e^{-\delta'\big(\zeta' \frac{\|\vecn\|_{\scaleto{1}{3pt}}}{K}\wedge\, \beta'\big)K} + (s+t_1+1)\e^{-\theta' K}
+ \e^{-\frac{a(s-t_{\scaleto{1}{3pt}})}{\log K}} + \e^{-\frac{a t_{\scaleto{1}{3pt}}}{\log K}}
\right).
\end{align*}
Now we have
\begin{align*}
\MoveEqLeft \frac{\scaleto{2}{6pt}}{T^{\scaleto{2}{4pt}}} \left|\int_{0}^{T-1}  \dd s\int_{0}^{s}\E_{\vecn}
\Big[f\big(\leproc(t_{1})\big)\, \E_{\leproc(t_{1}) } \big[ g\big(\leproc(s-t_1)\big)\big]\Big]\dd t_{1}- \nuk(f)^{\scaleto{2}{4pt}}\right|\\
& \leq \Oun K^q (c_q\norm{\vecn}^{\scaleto{q}{4pt}} +K^q) 
\left( 
\frac{1}{T}+ \e^{-\delta'\big(\zeta' \frac{\|\vecn\|_{\scaleto{1}{3pt}}}{K}\wedge\, \beta'\big)K} + T \e^{-\theta' K} + \frac{\log K}{T}
\right).
\end{align*}
The final result for $T\geq 1$ follows by collecting all estimates. For $T<1$ the bound follows directly from Proposition \ref{tempscourts}.

\subsection{Starting from the qsd: proof of Corollary \ref{varcornu}}

The result follows from Proposition \ref{varcorn} by integrating over $\vecn$ with respect to the qsd.
More precisely, we have
\begin{align*}
\MoveEqLeft \E_{\vecn}\Big[\big|S_{f}(T,K)-\nuk(f)\big|^{\scaleto{2}{4pt}}\Big]\leq
C'\|f\|_{\sK\!,q}^{\scaleto{2}{4pt}} \bigg(\frac{(c_q\norm{\vecn}^{\scaleto{q}{4pt}} +K^q) \norm{\vecn}^{\scaleto{q}{4pt}} +K^q \log K}{T\vee 1}\\
& +(c_q\norm{\vecn}^{\scaleto{q}{4pt}} +K^q) K^q \e^{-\delta'\big(\zeta' \frac{\|\vecn\|_{\scaleto{1}{3pt}}}{K}\wedge\, \beta'\big)K} 
+ (c_q\norm{\vecn}^{\scaleto{q}{4pt}} +K^q)T K^q \e^{-\theta'K} 
\bigg).
\end{align*}
The integrals of the first and third terms with respect to the q.s.d are estimated using Corollary \ref{bornepuissance}.
We deal with second term:
\begin{align*}
\MoveEqLeft \int  (c_q\norm{\vecn}^{\scaleto{q}{4pt}} +K^q) K^q  \e^{-\delta'\big(\zeta' \frac{\|\vecn\|_{\scaleto{1}{3pt}}}{K}\wedge\, \beta'\big)K} \dd\nuk(\vecn)=\\
& 
\int \un_{\{\{\norm{\vecn} < \beta'K/\zeta'\}\}\cap\, \mathcal{D}\}} (c_q\norm{\vecn}^{\scaleto{q}{4pt}} +K^q) K^q
\e^{-\delta'\big(\zeta' \frac{\|\vecn\|_{\scaleto{1}{3pt}}}{K}\wedge\, \beta'\big)K} \dd\nuk(\vecn)\\
& + 
\int  \un_{\{\{\norm{\vecn} < \beta'K/\zeta'\}\}\cap\, \mathcal{D}^c\}} (c_q\norm{\vecn}^{\scaleto{q}{4pt}} +K^q) K^q
\e^{-\delta'\big(\zeta' \frac{\|\vecn\|_{\scaleto{1}{3pt}}}{K}\wedge\, \beta'\big)K} \dd\nuk(\vecn)\\
& +
\int  \un_{\{\norm{\vecn} \geq \beta'K/\zeta'\}} (c_q\norm{\vecn}^{\scaleto{q}{4pt}} +K^q) K^q
\e^{-\delta'\big(\zeta' \frac{\|\vecn\|_{\scaleto{1}{3pt}}}{K}\wedge\, \beta'\big)K} \dd\nuk(\vecn).
\end{align*}
The third integral is estimated using the fact that the integrand is exponentially small in $K$.
The second integral is estimated using the first estimate in Corollary \ref{bornepuissance}.
We finally deal with the first integral. If $\vecn\in \mathcal{D}$ then
$\norm{\vecn}\geq \normdeux{\vecn}\geq \normdeux{\vecnf}\scaleto{/2}{9pt}$.
If $\{\|\vecn\|_{\scaleto{1}{4pt}} < \beta'K/\zeta'\}\cap \mathcal{D} \neq \emptyset$, on this set we have
$\e^{-\delta'\big(\zeta' \frac{\|\vecn\|_{\scaleto{1}{3pt}}}{K}\wedge\, \beta'\big)K}\leq \e^{-\delta'\zeta' \frac{\|\svecnf\|_{\scaleto{2}{3pt}}}{2}}$
(exponentially small in $K$). The estimate follows.

\section{Counting the number of births}\label{nbsauts}

Denote by $\mathcal{N}^{\sK}_\ell(t_{1},t_{2})$ the number of births of species of type $\ell$ between the times $t_{1}$ and $t_{2}$ ($1\leq \ell \leq d$, $0\leq t_1\leq t_2$). 

\begin{proposition}\label{prop-rentree}
For any probability measure $\mathfrak{m}$ on $\domaine$, we have
\[
\E_{\mathfrak{m}}\big[\mathcal{N}^{\sK}_\ell(t_{1},t_{2}) \big]=
K \int_{t_{1}}^{t_{2}} \E_{\mathfrak{m}}\!\left[B_{\ell}\left(\frac{\leproc(s)}{K}\right)\right] \dd s
\] 
and
\begin{align*}
\MoveEqLeft \E_{\mathfrak{m}}\bigg[\Big(\mathcal{N}^{\sK}_\ell(t_{1},t_{2})
 - \E_{\mathfrak{m}}\big[\mathcal{N}^{\sK}_\ell(t_{1},t_{2})\big] \Big)^{\scaleto{2}{4pt}}\bigg]
\leq  2K\E_{\mathfrak{m}}\!\left[ \,\int_{t_{1}}^{t_{2}} B_{\ell}\left(\frac{\leproc(s)}{K}\right) \dd s \right]\\
& \qquad\qquad+ \E_{\mathfrak{m}} \!\left[\bigg(\int_{t_{1}}^{t_{2}}  K B_{\ell}\left(\frac{\leproc(s)}{K}\right)  \dd s
- \E_{\mathfrak{m}}\big[\,\mathcal{N}^{\sK}_\ell(t_{1},t_{2})\big] \Big)^{\scaleto{2}{4pt}}\right].
\end{align*}
\end{proposition}

\begin{proof}
Recall that the generator of the process is given in \eqref{gene}. Let us now give a pathwise representation of the process.  We introduce $d$ 
independent point Poisson measures $M_{\ell}(\!\dd s, \!\dd\theta)$ on $\mathds{R}_{\scaleto{+}{4.5pt}}^{\scaleto{2}{4.5pt}}$ with intensity
$\dd s \dd\theta$. We define the $d$-dimensionnal c\`ad-l\`ag process $(N_{t}, t\in\mathds{R}_{\scaleto{+}{4.5pt}})$
\begin{align*}
\MoveEqLeft  N_{t} = N_{0} + \sum_{\ell = 1}^d \int_{0}^t \int M_{\ell}(\!\dd s,\! \dd\theta) \\
&  \times \left(\un_{\left\{\theta\,\,\leq K B_{\ell}\left(\frac{\sleproc(s)}{K}\right)\right\}}  -
\un_{\left\{K B_{\ell}\left(\frac{\sleproc(s)}{K}\right)\leq\, \theta\,\leq K \left(B_{\ell}\left(\frac{\sleproc(s)}{K}\right) + 
D_{\ell}\left(\frac{\sleproc(s)}{\sK}\right)\right)\right\}} \right) .
\end{align*}
Then the number of births of species of type $\ell$ occuring between the times $t_{1}$ and $t_{2}$ is given by
\[
\mathcal{N}^{\sK}_\ell(t_{1},t_{2}) =  \int_{t_{1}}^{t_{2}} \int \un_{\left\{\theta\,\leq K B_{\ell}\left(\frac{\sleproc(s)}{K}\right)\right\}} 
M_{\ell}(\!\dd s, \!\dd\theta).
\]
Using the Markov property we get at once the first identity.\newline 
We now establish the estimate. Indeed
\begin{align*}
&\E_{\mathfrak{m}}\left[\Big(\mathcal{N}^{\sK}_\ell(t_{1},t_{2}) - \E_{\mathfrak{m}}\big(\mathcal{N}^{\sK}_\ell(t_{1},t_{2}) \Big)^{\scaleto{2}{4pt}}\right]
\\ 
& \leq 2\E_{\mathfrak{m}}\!\left[\Big(\mathcal{N}^{\sK}_\ell(t_{1},t_{2}) - \int_{t_{1}}^{t_{2}} 
K B_{\ell}\left(\frac{\leproc(s)}{K}\right) \dd s\Big)^{\scaleto{2}{4pt}}\right]\\
&  \quad + 2\E_{\mathfrak{m}}\! \left[\Big(\int_{t_{1}}^{t_{2}}  K B_{\ell}\left(\frac{\leproc(s)}{K}\right) \dd s -
\E_{\mathfrak{m}}\big[\,\mathcal{N}^{\sK}_\ell(t_{1},t_{2})\big] \Big)^{\scaleto{2}{4pt}} \right].
\end{align*}
By the $\mathds{L}^2$-isometry for jump processes (see \cite[Formula (3.9) p.62]{Ikeda}), we have 
\begin{align*} 
\MoveEqLeft \E_{\mathfrak{m}}\!\left[\Big(\mathcal{N}^{\sK}_\ell(t_{1},t_{2}) - \int_{t_{1}}^{t_{2}}  K B_{\ell}\left(\frac{\leproc(s)}{K}\right) \dd s\Big)^{\scaleto{2}{4pt}}\right]\\
&=  \int_{t_{1}}^{t_{2}} \int \E_{\mathfrak{m}}\bigg[\bigg(\un_{\left\{\theta\leq K B_{\ell}\big(\frac{\sleproc(s)}{K}\big)\right\}}\bigg)^{\scaleto{2}{4pt}} \bigg]\dd s \dd \theta \\
&= \int_{t_{1}}^{t_{2}} \E_{\mathfrak{m}}\left[ K B_{\ell}\left(\frac{\leproc(s)}{K}\right) \right]\dd s.
\end{align*}
This finishes the proof.
\end{proof}

\section{Gaussian limit for the rescaled qsd}

We have the following theorem of independent interest. A part of this theorem partially generalizes a result obtained
in \cite{ccm1} for models involving a single species ($d=1$). Recall that $\vecnf = \lfloor K\vecxf \rfloor$.
\begin{theorem}
\label{confinement}
For all $K>1$, define the measure $\mathfrak{a}_K$ on the Borel $\sigma$-algebra of $\real^d$ by
\[
\mathfrak{a}_{\sK}(\cdot)=\nuk\left(\left\{ \vecn\in \domaine : \frac{\vecn-\vecnf}{\sqrt{K}}\in \cdot\right\}\right).
\]
Then $(\mathfrak{a}_{\sK})_{\sK}$ converges weakly to the centered Gaussian measure with covariance matrix 
\[
\EuScript{S}=\int_{0}^{\infty}\e^{\tau M^*}\EuScript{B}^*\e^{\tau {M^*}^{\intercal}}\dd\tau.
\]
where $\EuScript{B}^*$ is the diagonal matrix with entries $B_{\ell}(\vecxf)=D_{\ell}(\vecxf)$.
The matrix $\EuScript{S}$ is also the unique symmetric solution of the (Lyapunov)
equation (fluctuation-dissipation relation)
\begin{equation}\label{pierre}
M^* \EuScript{S}+\EuScript{S} {M^*}^{\intercal}=-\EuScript{B}^*\,.
\end{equation}
\end{theorem}

\begin{remark}
We have 
\[
\lim_{\sK\to+\infty} \frac{\Sigma^{\sK}}{K}=\EuScript{S}.
\]
This follows by dividing out equation \eqref{kuboeq} by $K$, letting $K$ tend to infinity, and using the uniqueness
of the (symmetric) solution of \eqref{pierre}.
\end{remark}

\begin{proof}
By Theorem \ref{bornemom}, the family of measures $(\mathfrak{a}_{\sK})_{\sK}$ is tight.
For $\vp\in \real^d$ define
\[
H_{\sK}(\hspace{1pt}\vp)=\int \e^{\ii \frac{\langle\, \vp,(\vecn-\svecnf)\rangle}{\sqrt{K}}}\dd \nuk(\vecn).
\]
It follows also from Theorem \ref{bornemom} that the family of functions $(H_{\sK})$ is uniformly bounded in $C^{2}$.
We will prove that 
\begin{equation}\label{pelouse}
\lim_{\sK\to\infty}H_{\sK}(\hspace{1pt}\vp)=\e^{-\langle\, \vp,\hspace{1pt}\EuScript{S}\vp\rangle},\;\forall \vp\in \mathds{R}^d.
\end{equation}
This will entail that there is only one weak accumulation point for $(\mathfrak{a}_{\sK})_{\sK}$.
The proof will be the consequence of Prokhorov Theorem \cite{patrick}.
Using \eqref{eq:miracle} and \eqref{eq:lambda0}, we have
\[
\lim_{\sK\to\infty}\nuk\!\left(\gen \e^{\ii\frac{\langle\, \vp,(\, \cdot-\svecnf)\rangle}{\sqrt{K}}}\right)=0\,.
\]
We also have
\begin{align*}
\MoveEqLeft[4] \nuk\!\left(\gen \e^{\ii \frac{\langle\, \vp,(\,\cdot\,-\svecnf)\rangle}{\sqrt{K}}}\right)
=
K\sum_{\ell=1}^{d}\int\dd\nuk(\vecn)\e^{\ii\frac{\langle\, \vp,(\vecn-\svecnf)\rangle}{\sqrt{K}}} \\
& \times \left(B_{\ell}\left(\frac{\vecn}{K}\right)\left(\e^{\ii\frac{p_{\ell}}{\sqrt{K}}}-1\right)
+ D_{\ell}\left(\frac{\vecn}{K}\right)\left(\e^{-\ii\frac{p_{\ell}}{\sqrt{K}}}-1\right)\right).
\end{align*}
Using Taylor expansion, and  the moments estimates and the polynomial bounds on $B_{\ell}$ and $D_{\ell}$
(and $B_{\ell}(\vecxf)=D_{\ell}(\vecxf)$) we obtain
\begin{align*}
\MoveEqLeft \nuk\!\left(\gen \e^{\ii\frac{\langle\, \vp,(\,\ushort{\cdot}\,-\svecnf)\rangle}{\sqrt{K}}}\right)
\\
& =-\sum_{\ell=1}^{d}B_{\ell}\left(\frac{\vecnf}{K}\right) p_{\ell}^{\scaleto{2}{4pt}}\,H_{\sK}(\vp)+ \ii\sum_{\ell=1}^{d} p_{\ell}\sum_{j=1}^{d}
\left(\partial_{j}B_{\ell}\left(\frac{\vecnf}{K}\right)-\partial_{j}D_{\ell}\left(\frac{\vecnf}{K}\right)\right)\\
& \qquad \times \int \e^{\ii\frac{\langle\, \vp, (\vecn-\svecnf)\rangle}{\sqrt{K}}}
\frac{n_{j}^{\scaleto{*}{2.5pt}}-n_{j}^{\scaleto{*}{2.5pt}}}{\sqrt{K}} \dd\nuk(\vecn)
+\mathcal{O}\left(\frac{1}{\sqrt{K}}\right)\\
& = -\sum_{\ell=1}^{d}B_{\ell}\left(\frac{\vecnf}{K}\right) p_{\ell}^{\scaleto{2}{4pt}}\, H_{\sK}(\vp)\\
& \qquad +\sum_{\ell=1}^{d}p_{\ell}
\sum_{j=1}^{d}\left(\partial_{j}B_{\ell}\left(\frac{\vecnf}{K}\right)-\partial_{j}D_{\ell}\left(\frac{\vecnf}{K}\right)\right)\partial_{p_{j}}H_{\sK}(\vp)+\mathcal{O}\left(\frac{1}{\sqrt{K}}\right)\\
& =-\sum_{\ell=1}^{d}B_{\ell}(\vecxf)\,p_{\ell}^{\scaleto{2}{4pt}}\,H_{\sK}(\vp)+\sum_{\ell=1}^{d} p_{\ell}\sum_{j=1}^{d}
M_{\ell,j}^*\,\partial_{p_{j}}H_{\sK}(\vp)+\mathcal{O}\left(\frac{1}{\sqrt{K}}\right).
\end{align*}
We conclude that every accumulation point $\breve{H}$ of $(H_{\sK})_{\sK}$ is bounded in
$C^{1}$, satisfies $\breve{H}(\ushort{0})=1$, and is a solution of the equation
\[
-\sum_{\ell=1}^{d}B_{\ell}(\vecxf)\,p_{\ell}^{\scaleto{2}{4pt}}\,\breve{H}(\vp)+\sum_{\ell=1}^{d} p_{\ell}\sum_{j=1}^{d}
M_{\ell,j}^*\,\partial_{p_{j}}\breve{H}(\vp)=0.
\]
Then \eqref{pelouse} follows from Lemma \ref{uniquesol} (stated and proved right after this proof) with $A=M^*$. 
\end{proof}

\begin{lemma}\label{uniquesol}
Let $(B_{j})$ be $d$ strictly positive numbers and $A$ a real $d\times
d$ matrix
such that $\mathrm{Sp}(A)\subset\{z\in\mathds{C}: \operatorname{Re}(z)<0\}$. Then 
there exists a unique $C^{1}(\mathds{R}^d,\mathds{R})$ function $H$ satisfying
$H(\ushort{0})=1$ and
\begin{equation}
\label{pde}
-\sum_{\ell=1}^{d}B_{\ell}\, p_{\ell}^{\scaleto{2}{4pt}}\,H(\hspace{1pt}\vp)+\sum_{\ell=1}^{d} p_{\ell}\sum_{j=1}^{d}A_{\ell,\,j}\,\partial_{p_{j}}H(\hspace{1pt}\vp)=0,\; \vp\in \mathds{R}^d.
\end{equation}
This function is given by
\[
H(\hspace{1pt}\vp)=\e^{-\langle\, \vp,\EuScript{S}\vp \rangle}
\]
where
\[
\EuScript{S}=\int_{0}^{\infty}\e^{\tau A}\EuScript{B}\e^{\tau A^{\intercal}}\dd\tau
\]
where $\EuScript{B}$ is the diagonal matrix with entries $(B_{j})$.
The matrix $S$ is also the unique symmetric solution of the
equation 
\[
A \EuScript{S}+\EuScript{S} A^{\intercal}=-\EuScript{B}.
\]
\end{lemma}
\begin{proof} We use the method of characteristics.  For all $\vp\in \mathds{R}^d$, we define the function
$\vp(s), s\geq 0$ as the solution of 
\[
\frac{\dd\vp}{\dd s}(s)= A^{\intercal}\vp(s),\;\;\vp(0)=\vp.
\]
Let
\[
\mathfrak{b}(s)=-\sum_{\ell=1}^d B_{\ell} \int_{0}^{s} p_{\ell}(\tau)^{\scaleto{2}{4pt}}\dd\tau.
\]
Let $H$ be a solution of \eqref{pde}. It is easy to check that for all $\vp\in\rpd$ and $s\in\real$ 
\[
\frac{\dd}{\dd s}\left(H\!\left(\hspace{1pt}\vp(s)\right)\e^{\mathfrak{b}(s)}\right)=0.
\]
Integrating from $0$ to $u$ yields
\[
H(\hspace{1pt}\vp)=H\!\left(\hspace{1pt}\vp(u)\right)  \e^{\mathfrak{b}(u)}.
\]
From the spectral properties of $A$ we get
\[
\lim_{u\to+\infty}H\!\left(\hspace{1pt}\vp(u)\right) =H(\ushort{0})=1.
\]
Therefore
\[
H(\hspace{1pt}\vp)=\e^{\mathfrak{b}(\infty)}
\]
and
\[
\mathfrak{b}(\infty)=-\int_{0}^{\infty}\left\langle\,\vp\hspace{1pt},\e^{\tau A}\EuScript{B}\e^{\tau A^{\intercal}}\vp\right\rangle\dd\tau
=-\langle\, \vp\hspace{1pt},\EuScript{S}\vp\rangle.
\]
Finally we get from the spectral properties of $A$
\begin{align*}
A \EuScript{S}+\EuScript{S} A^{\intercal}
&= \int_{0}^{\infty}\left(A\e^{\tau A}\EuScript{B}\e^{\tau A^{\intercal}}+\e^{\tau A}\EuScript{B}\e^{\tau A^{\intercal}} A^{\intercal}\right)
\dd\tau\\
&= \int_{0}^{\infty}\frac{\dd}{\dd\tau}\left(\e^{\tau A}\EuScript{B}\e^{\tau A^{\intercal}}\right)\dd\tau
=-\EuScript{B}.
\end{align*}
This finishes the proof of the lemma.
\end{proof}

\bigskip \noindent {\bf Acknowledgements:}  We thank the two anonymous referees for fruitful comments and suggestions.

\bigskip \noindent {\bf Funding:} 
S. M. has been supported by the Chair ``Mod\'elisation Math\'ematique et Biodiversit\'e'' of Veolia Environnement-Ecole Polytechnique-Museum national d'Histoire naturelle-Fondation X. P. C. and S. M. warmly thank the Basal Conicyt CMM AFD170001 
project.  J.-R. C. and P. C.  also acknowledge the hospitality of the Instituto de F\'{\i}sica de San Luis Potos\'{\i}.



\begin{thebibliography}{99} 

\bibitem{patrick}
P. Billingsley.
\emph{Convergence of probability measures.}
Second edition. Wiley Series in Probability and Statistics: Probability and Statistics. John Wiley \& Sons, 1999. 

\bibitem{ccm1}
J.-R. Chazottes, P. Collet, S. M\'el\'eard.
Sharp asymptotics for the quasi-stationary distribution of birth-and-death processes.
Probab. Theory Related Fields, Vol. 164 (2016), no. 1-2, 285--332.
 
\bibitem{ccm2} 
J.-R. Chazottes, P. Collet, S. M\'el\'eard.
On time scales and quasi-stationary distributions for multitype birth-and-death processes.
Annales de l'Institut Henri Poincar\'e - Probabilit\'es et statistiques, Vol. 55 (2019), no 4, 2249--2294. 

\bibitem{EK}
S. N. Ethier, T. G. Kurtz.
\emph{Markov processes, Characterization and convergence.}
Wiley \& Sons, 1986.

\bibitem{resilience-book} 
L.H. Gunderson, C.R. Allen, C.S. Holling (Eds).
\emph{Foundations of Ecological Resilience}. Island Press, 2012.

\bibitem{Ikeda}
N. Ikeda, S. Watanabe. 
\emph{Stochastic differential equations and diffusion processes.} 
North Holland 2nd edition, 1989.
  
\bibitem{kubo}
R. Kubo. The fluctuation-dissipation theorem. 
Rep. Prog. Phys. \textbf{29}, 255--284 (1966).
  
\bibitem{sim}
V. Simoncini.
Computational methods for linear matrix equations. SIAM Review 58 (2016), no. 3, 377--441.
 
\end{thebibliography}
 \end{document}